\documentclass[12pt]{amsart}
\usepackage{amsmath}
\usepackage{amsfonts}
\usepackage{amssymb,color}
\usepackage{hyperref}
\usepackage{amsthm,doi}
\usepackage{thmtools}
\usepackage{physics}
\usepackage{tikz}
\usepackage{pgfplots}
\pgfplotsset{compat=1.16}

\hypersetup{
	colorlinks   = true, 
	urlcolor     = blue, 
	linkcolor    = blue, 
	citecolor   = red 
}

\makeatletter
\newcommand\RedeclareMathOperator{%
	\@ifstar{\def\rmo@s{m}\rmo@redeclare}{\def\rmo@s{o}\rmo@redeclare}%
}
\newcommand\rmo@redeclare[2]{%
	\begingroup \escapechar\m@ne\xdef\@gtempa{{\string#1}}\endgroup
	\expandafter\@ifundefined\@gtempa
	{\@latex@error{\noexpand#1undefined}\@ehc}%
	\relax
	\expandafter\rmo@declmathop\rmo@s{#1}{#2}}
\newcommand\rmo@declmathop[3]{%
	\DeclareRobustCommand{#2}{\qopname\newmcodes@#1{#3}}%
}
\@onlypreamble\RedeclareMathOperator
\makeatother

\RedeclareMathOperator{\var}{Var}
\DeclareMathOperator{\cov}{Cov}
\DeclareMathOperator{\sgn}{sgn}
\DeclareMathOperator{\jac}{Jac}

\renewcommand{\ip}[2]{\ensuremath{\left<#1,#2\right>}}

\newcommand{\bigabs}[1]{\big\lvert #1 \big\rvert}
\newcommand{\biggabs}[1]{\bigg\lvert #1 \bigg\rvert}
\newcommand{\funct}{\varphi}
\newcommand{\func}{\varphi}

\newcommand{\charge}{\kappa}
\newcommand{\newcharge}{\mu}
\newcommand{\vecc}{(z_1, \ldots, z_n)}
\newcommand{\RR}{\mathbb{R}}
\newcommand{\tfs}{\boldsymbol{\Phi}}

\newcommand{\E}{\mathbb{E}}
\newcommand{\per}{\mathrm{per}} 
\newcommand{\newdelta}{\Delta_H}

\newcommand{\chsign}{\charge}
\newcommand{\noise}{\mathcal{N}}

\newtheorem{lemma}{Lemma}[section]
\newtheorem{theorem}[lemma]{Theorem}
\newtheorem{coro}[lemma]{Corollary}
\newtheorem{prop}[lemma]{Proposition}

\newtheorem{definition}[lemma]{Definition}

\theoremstyle{definition}
\newtheorem{example}[lemma]{Example}

\numberwithin{equation}{section}
\newtheorem{rem}[lemma]{Remark}

\allowdisplaybreaks

\author[A. Haimi]{Antti Haimi}
\address[A. H.]{Faculty of Mathematics, University of Vienna,
	Oskar-Morgenstern-Platz 1, A-1090 Vienna, Austria}
\email{antti.haimi@univie.ac.at}

\author[G. Koliander]{G\"{u}nther Koliander}
\address[G. K.]{Acoustics Research Institute, Austrian Academy of Sciences,
Wohllebengasse 12-14 A-1040, Vienna, Austria}
\email{gkoliander@kfs.oeaw.ac.at}

\author[J. L. Romero]{Jos\'{e} Luis Romero}
\address[J. L. R.]{Faculty of Mathematics, University of Vienna,
Oskar-Morgenstern-Platz 1, A-1090 Vienna, Austria\\
and
Acoustics Research Institute, Austrian Academy of Sciences, Wohllebengasse
12-14 A-1040, Vienna, Austria}
\email{jose.luis.romero@univie.ac.at}

\thanks{A.\ H., G.\ K. and J.\ L.\ R. gratefully acknowledge support from Austrian Science Fund (FWF): Y 1199, P 31153, and P 29462, and from the WWTF grant INSIGHT (MA16-053). Preliminary versions of this work were presented by A.\ H. at the meeting of the mathematics group of the Acoustic Research Institute of the Austrian Academy of Sciences during 2018 and 2019. We kindly thank the institute members for their helpful comments.}



\title{Zeros of Gaussian Weyl-Heisenberg functions and
hyperuniformity of charge}

\begin{document}\setlength{\jot}{10pt}
	
\begin{abstract}
We study Gaussian random functions on the complex plane whose stochastics are invariant under the Weyl-Heisenberg group (twisted stationarity). The theory is modeled on translation invariant Gaussian entire functions, but allows for non-analytic examples, in which case winding numbers can be either positive or negative.

We calculate the first intensity of zero sets of such functions, both when considered as points on the plane, or as charges according to their phase winding. In the latter case, charges are shown to be in a certain average equilibrium independently of the particular covariance structure (universal screening). We investigate the corresponding fluctuations, and show that in many cases they are suppressed at large scales (hyperuniformity). This means that universal screening is empirically observable at large scales. We also derive an asymptotic expression for the charge variance.

As a main application, we obtain statistics for the zero sets of the short-time Fourier transform of complex white noise with general windows, and also prove the following uncertainty principle: the expected number of zeros per unit area is minimized, among all window functions, exactly by generalized Gaussians. Further applications include poly-entire functions such as covariant derivatives of Gaussian entire functions.
\end{abstract}

\maketitle

\section{Introduction and Results}
The investigation of zeros of random functions with Gaussian distribution is a classical endeavor in statistical physics, in great part motivated by the goal to derive generic models for the distribution of quantum chaotic systems \cite{MR489542,MR1913853}. This article is concerned with complex-valued random functions on the complex plane, in which case zeros are typically discrete, and correspond to phase singularities \cite{berry2000phase}. While random functions on the Euclidean plane are often studied under the assumption of \emph{stationarity}, that is, stochastic invariance under Euclidean shifts, we study a form of invariance compatible with the complex structure of the plane, which we call \emph{twisted stationarity}.

A model case for our theory are random power series with properly scaled independent normally distributed coefficients. Such random waves are known as translation invariant Gaussian entire functions (TI-GEF), because, even though they are not stochastically invariant under Euclidean shifts, their zeros are \cite{NSwhat, gafbook}. In fact, due to the so-called \emph{Calabi rigidity}, these are the only examples of Gaussian \emph{analytic} functions on the plane with stationary zero sets  \cite{NSwhat}.

The notion of twisted stationarity that we introduce abstracts some properties of TI-GEF such as stationarity of zeros, while, crucially, allowing for non-analytic examples. Indeed, our main motivation is the study of certain possibly non-analytic random functions such as the
correlation of white noise with the time-frequency shifts of a given reference window function (short-time Fourier transform or cross radar ambiguity \cite[Chapter 1]{folland89} \cite[Chapter 3]{charlybook}). While the freedom to choose a reference window function is very valuable in applications, only one specific choice leads to Gaussian entire functions \cite[Chapter 15]{flandrin2018explorations}, \cite{MR4047541,bh}. Further motivation comes from the fact that an operation as basic as computing a derivative of a TI-GEF (in the sense of complex geometry) does not preserve analyticity. Our starting point is the identification of a common element in the previous examples: stochastic invariance under a certain representation of the \emph{Weyl-Heisenberg group}. Twisted stationarity provides a unified model for such situation, and the corresponding random functions are called Gaussian Weyl-Heisenberg functions (GWHF).

The zeros of Gaussian entire functions are statistically rich. In contrast to points in a Poisson process, they exhibit \emph{repulsion}, that is, negative correlation similar to that of charged particles of equal sign, and they are \emph{hyperuniform}, in the sense that the variance of the number of points in a large observation window is asymptotically smaller than the corresponding expected value \cite{torquato2016hyperuniformity, ghosh2017fluctuations}.
In the non-analytic setting of GWHF we shall find the new element of a \emph{signed charge}, since, as is the case with non-analytic random waves, zeros may have negative winding numbers \cite{berry2000phase}.

We now introduce the main mathematical objects and results, and provide context on their significance.

\subsection{Gaussian Weyl-Heisenberg Functions}
We study zero sets of Gaussian circularly symmetric random functions on the plane $F\colon \mathbb{C} \to \mathbb{C}$ whose covariance kernel
is given by \emph{twisted convolution}:
\begin{align}
\label{eq_twisted_covariance_kernel}
\mathbb{E}
\left[
F(z) \cdot \overline{F(w)}
\right]=
H(z-w) \cdot e^{i \Im(z \overline{w})}, \qquad z,w \in \mathbb{C}.
\end{align}
Here, $H\colon \mathbb{C} \to \mathbb{C}$ is a function called
 \emph{twisted kernel}.
Gaussianity means that for each $z_1, \ldots, z_n \in \mathbb{C}$, 
$(F(z_1), \ldots, F(z_n))$ is a normally distributed complex random vector. Circularity means that $F \sim e^{i \theta} F$, for all $\theta \in \mathbb{R}$,
and implies that $F$ has vanishing expectation and pseudo-covariance, i.e.,
$\mathbb{E}
\left[
F(z)\right]=0$, $\mathbb{E}
\left[
F(z) F(w)
\right]=0$, for all $z,w \in \mathbb{C}$. Hence, the stochastics of $F$ are completely encoded in the twisted kernel \eqref{eq_twisted_covariance_kernel}.

While the covariance structure \eqref{eq_twisted_covariance_kernel} without the  complex exponential factor would mean that $F$ is stationary, the presence of the oscillatory factor means that $F$ is \emph{twisted stationary}:
\begin{align}\label{eq_twisted_stationairy}
\mathbb{E} \left[ e^{i\Im(z \overline{\zeta})} \cdot F(z-\zeta)
\cdot
\overline{ e^{i\Im(w \overline{\zeta})} \cdot F(w-\zeta)} \right]
=
\mathbb{E} \left[F(z) \cdot \overline{F(w)} \right],
\qquad z,w, \zeta \in \mathbb{C}.
\end{align}
In other words, the stochastics of $F$ are invariant under \emph{twisted shifts}:
\begin{align}\label{eq_intro_invariance}
F(z) \mapsto e^{i\Im(z \overline{\zeta})} \cdot F(z-\zeta),\qquad \zeta \in \mathbb{C}.
\end{align}
We call such a random function $F$ a \emph{Gaussian Weyl-Heisenberg function} (GWHF), as the operators \eqref{eq_intro_invariance} generate the (reduced) Weyl-Heisenberg group \cite[Chapter 1]{folland89}. Let us mention some motivating examples (which are developed in more detail in Section \ref{sec_app}).

\begin{example}[\emph{Gaussian entire functions}]\label{ex_gaf}

Let $F$ be a GWHF with twisted kernel $H(z)=e^{-\tfrac{1}{2}\abs{z}^2}$ and set $G(z)=e^{\tfrac{1}{2}\abs{z}^2} F(z)$. Then
\begin{align*}
\mathbb{E}
\left[
G(z) \cdot \overline{G(w)}
\right]= \exp\Big[\tfrac{1}{2}\abs{z}^2+\tfrac{1}{2}\abs{w}^2
-\tfrac{1}{2}\abs{z-w}^2 + i \Im(z\bar{w})\Big]
=e^{z \bar{w}}.
\end{align*}
Hence, $G$ is a Gaussian entire function on the plane
with correlation kernel given by the \emph{Bargmann-Fock kernel} \cite{zhu}, and its zero set is well-studied \cite{gafbook}. In terms of $G$, the twisted stationarity property of $F$ \eqref{eq_intro_invariance} is an instance of the \emph{projective invariance} property \cite{NSwhat}, and, indeed, reflects the invariance of the stochastics of $G$ under \emph{Bargmann-Fock shifts}:
\begin{align}\label{eq_bfs}
G(z) \mapsto e^{-\frac{1}{2} |\zeta|^2 + z \overline{\zeta}} \cdot G(z-\zeta), \qquad \zeta \in \mathbb{C}.
\end{align}
\end{example}

\begin{example}[\emph{The short-time Fourier transform of complex white noise}]\label{ex_stft}
	
Given a \emph{window function} $g \in \mathcal{S}(\mathbb{R})$,
the \emph{short-time Fourier transform} of a function $f\colon \mathbb{R} \to \mathbb{C}$ is 
\begin{align}
\label{eq_stft}
V_g f(x,y) = \int_{\mathbb{R}} f(t) \overline{g(t-x)} e^{-2\pi i t y} dt,
\qquad (x,y) \in \mathbb{R}^2.
\end{align}
The short-time Fourier transform is a windowed Fourier transform, and the value $V_g f(x,y)$ represents the influence of the frequency $y$ near $x$.
As localizing window $g$, one often chooses Gaussian,
or, more generally, Hermite functions, as these optimize several measures related to Heisenberg's uncertainty principle.

In signal processing, the short-time Fourier transform is often used to analyze functions (called signals) contaminated with random noise $\mathcal{N}$. The corresponding zero sets play an important role in many modern algorithms, for example in the dynamics of certain non-linear procedures to sharpen spectrograms \cite[Chapter 12]{flandrin2018explorations} or in the design of filtering masks, where landmarks are chosen guided by the statistics of zero sets \cite{flandrin2015time}.

Of particular interest are the zeros of the short-time Fourier transform of complex white Gaussian noise 
$V_g \, \mathcal{N}$ \cite[Chapter 15]{flandrin2018explorations}. While many applications demand the use of different window functions $g$ --- see, e.g., \cite[Section 10.2]{flandrin2018explorations} --- zero-statistics for the STFT are currently only understood for Gaussian windows \cite{MR4047541,bh}, as these facilitate the application of the theory of Gaussian entire functions (Example 1.1). One main motivation for this article is to obtain zero-statistics for general windows $g$, including for example Hermite functions. (Some related numerics can be found in \cite[Chapter 15]{flandrin2018explorations}.)

With an adequate distributional interpretation, the STFT of complex white Gaussian noise with respect to a Schwartz window function $g$ defines a smooth circularly symmetric Gaussian function on the plane. Twisted stationarity is revealed by the transformation
\begin{align}\label{eq_trans_stft}
F(x+iy) := e^{-i xy} \cdot V_g \, \mathcal{N} \big(x/\sqrt{\pi}, -y/\sqrt{\pi}\big),
\end{align}
which, as shown in Section \ref{sec_tf}, indeed yields a GWHF. 
The twisted stationarity of $F$ reflects the invariance of the stochastics of complex white noise under \emph{time-frequency shifts}
\begin{align*}
f(t) \mapsto e^{2 \pi i b t} f(t-a), \qquad (a,b) \in \mathbb{R}^2.
\end{align*}
Basic questions about zero sets of short-time Fourier transforms also underlie problems about the spanning properties of the time-frequency shifts of a given function (Gabor systems) \cite{jan03,MR3336091}, or about Berezin's quantization \cite{gjm20}. The study of random counterparts provides a first form of average case analysis for such problems.
\end{example}

\begin{example}[\emph{Derivatives of GEF}]\label{ex_true}
	The \emph{covariant derivative} of an entire function $G\colon \mathbb{C} \to \mathbb{C}$
	is
	\begin{align*}
	\bar{\partial}^* G(z) = \bar{z}\,G(z) - \partial G(z),
	\end{align*}
	and it is distinguished among other differential operators of order 1 because it commutes with the Bargmann-Fock shifts \eqref{eq_bfs}. As a consequence, if $G$ is a Gaussian entire function, as in Example \ref{ex_gaf}, the stochastics of $\bar{\partial}^* G$ are also invariant under Bargmann-Fock shifts, and the transformation
	\begin{align}\label{eq_trans}
	F(z)=e^{-\frac{1}{2}|z|^2} \bar{\partial}^* G(z)
	\end{align}
	yields a GWHF. The corresponding twisted kernel is computed in Section \ref{sec_pure}.	
	
	Zeros of covariant derivatives are instrumental in the description of vanishing orders of analytic functions \cite{bs93, ehr20}. They are also important in the study of weighted magnitudes of analytic functions $G$. For example, the \emph{amplitude} $A(z) = e^{-\frac{1}{2}|z|^2} |G(z)|$ of an entire function $G$ satisfies
	\begin{align}\label{eq_wa}
	\big|\nabla A \big| = e^{-\frac{1}{2}|z|^2} \big| \bar{\partial}^* G \big|.
	\end{align}
	Thus, the critical points of the amplitude of a Gaussian entire function $G$ are exactly the zeros of the GWHF \eqref{eq_trans} ---
	see also \cite{MR2104882,MR3937291}. The squared amplitude $A^2(z)$ is also of interest, as it corresponds after normalization to the \emph{spectrogram} of complex white noise with a Gaussian window (i.e., the squared absolute value of the STFT \eqref{eq_stft}) \cite{MR4047541}; see also 
	\cite[Corollary 2.8]{bh}.
\end{example}

\begin{example}[\emph{Gaussian poly-entire functions}]\label{ex_poly}
	Iterated covariant derivatives of an analytic function $G_0$,
	\begin{align}\label{eq_true}
	G = (\bar{\partial}^*)^{q-1} \, G_0,
	\end{align}
	are not themselves analytic, but satisfy a higher order 
	Cauchy-Riemann condition
	\begin{align}\label{eq_crq}
	\bar{\partial}\,^q \,G = 0,
	\end{align}
	known as poly-analyticity \cite{balk}. In Vasilevski's parlance \cite{vas00}, \eqref{eq_true} is a \emph{true or pure poly-entire function}, while the more general solution to \eqref{eq_crq},
	\begin{align}\label{eq_full}
	G = \sum_{k=0}^{q-1} \frac{1}{\sqrt{k!}} (\bar{\partial}^*)^{k} \,G_k,
	\end{align}
	with $G_0, \ldots, G_{q-1}$ entire, is a \emph{fully poly-entire function}.
	
	Random Gaussian poly-entire functions of either pure of full type are defined by \eqref{eq_true} and \eqref{eq_full}, letting $G_0, \ldots, G_{q-1}$ be independent Gaussian entire functions.
	
	Poly-entire functions are important in statistical physics, in the analysis of high energy systems of particles \cite{MR2593994}, and we expect their random analogs to also be useful in that field.
\end{example}

\subsection{Standing assumptions}\label{sec_sa}
The positive semi-definiteness of the covariance kernel of a GWHF $F$ reads as follows:\footnote{Here, $A \geq 0$ means that $A=B^*B$, for some square matrix $B$.}
\begin{align}
\label{eq_pd}
\Big(
H(z_k - z_j) \cdot e^{i \Im (z_k \overline{z_j})} \Big)_{j,k=1, \ldots, n} \geq 0 \qquad \mbox{for all } z_1,\ldots, z_n \in \mathbb{C}.
\end{align}
As a consequence, $H(-z) = \overline{H(z)}$ and $H(0) \geq 0$. To avoid trivial cases, we assume that $H(0) \not= 0$, since, otherwise, $F$ would be almost surely constant. We furthermore impose the normalization
\begin{align}\label{eq_HH}
H(0)=1.
\end{align}
We also assume that
\begin{align}\label{eq_non_deg}
\abs{H(z)} < 1, \qquad z \in \mathbb{C} \setminus \{0\},
\end{align}
which means that no two samples $F(z), F(w)$ with $z \not= w$ are deterministically correlated, as \eqref{eq_non_deg} amounts to the invertibility of the joint covariance matrix
$\left[
\begin{smallmatrix}
1 & H(z-w) e^{i\Im(z\bar{w})}\\
\overline{H(z-w)} e^{-i\Im(z\bar{w})} & 1
\end{smallmatrix}\right]$.

We also assume a certain regularity of the twisted kernel:
\begin{align}\label{eq_HC2}
\mbox{$H$ is $C^2$ in the real sense},
\end{align}
and denote the corresponding derivatives with supraindices;
e.g., $H^{(1,1)}(x+iy) = \partial_x \partial_y H(x+iy)$.
Finally, we will always assume that $F$ has $C^2$ paths in the real sense:
\begin{align}\label{eq_cont_paths}
\mbox{Almost every realization of $F$ is a $C^2(\mathbb{R}^2)$ function.}
\end{align}
This is the case, for example, if $H \in C^6(\mathbb{R}^2)$ in the real sense \cite[Theorem~5]{ghosal2006posterior}, but also other weaker assumptions suffice (see \cite[Theorem~1.4.2]{adler}). 

\begin{definition}
A function (twisted kernel) $H\colon \mathbb{C} \to \mathbb{C}$ is said to satisfy the standing assumptions if \eqref{eq_pd}, \eqref{eq_HH}, \eqref{eq_non_deg}, \eqref{eq_HC2}, and \eqref{eq_cont_paths} are fulfilled.
\end{definition}
Note that if $H$ satisfies the standing assumptions, then a GWHF $F$ with covariance kernel \eqref{eq_twisted_covariance_kernel} always exists \cite[Chapter 1]{level}. For short, we also say that \emph{$F$ is a GWHF satisfying the standing assumptions}.

\subsection{Zero sets}
We are mainly interested in the zero set of a GWHF $F$,
encoded in the random measure
\begin{align}\label{eq_z1}
\mathcal{Z}_F &:= \sum_{z \in \mathbb{C},\, F(z)=0} \delta_z \,,
\end{align}
where $\delta_z$ denotes the Dirac measure at $z$. This measure properly encodes the zero set of $F$, because, as we prove in Proposition \ref{prop_reg} below, under the standing assumptions the zeros of $F$ are almost surely simple and non-degenerate (i.e., as a map on $\mathbb{R}^2$, the differential matrix of $F$ is invertible).
	
Our first result describes the first point intensity of zero sets of GWHF.
\begin{theorem}[{First intensity of zero sets}]\label{th_unsigned_H}
	Let $F$ be a GWHF with twisted kernel $H$ satisfying the standing assumptions.	
	Then $\mathcal{Z}_F$ is a stationary random measure with first intensity:
	\begin{align}\label{eq_rho_H}
	\rho_1=\frac{1}{2\pi}  \frac{\newdelta+2}{\sqrt{\newdelta+1}},
	\end{align}
	where
	\begin{align}
	\label{eq_newdelta}
	\newdelta := \det
	\begin{bmatrix}
	-H^{(2,0)}(0)
	-\abs{H^{(1,0)}(0)}^2
	& -H^{(1,1)}(0)-i - H^{(1,0)}(0)H^{(0,1)}(0) \\
	-\overline{H^{(1,1)}(0)} +i - \overline{H^{(1,0)}(0)H^{(0,1)}(0)}
	& 
	-H^{(0,2)}(0)
	-\abs{H^{(0,1)}(0)}^2
	\end{bmatrix}.	
	\end{align}
	Concretely, for every Borel set $E \subseteq \mathbb{C}$:
	\begin{align}\label{eq_aaaaa}
	\mathbb{E} \big[ \# \big\{z \in E: F(z)=0 \big\}\big] = \rho_1 |E|.
	\end{align}
	In addition, $\newdelta \geq 0$, and therefore $\rho_1 \geq 1/\pi$.
\end{theorem}
In many important cases, the twisted kernel $H$ is radial, and the expression for the first point intensity can be simplified.

\begin{coro}\label{coro_one_point_radial}
	Let $F$ be as in Theorem~\ref{th_unsigned_H}.
	Assume further that 
	$H(z)=P\big(|z|^2\big)$,
	where $P\colon \mathbb{R} \to \mathbb{R}$
	is $C^2(\mathbb{R})$.
	Then $P'(0) \leq -1/2$ and the first point intensity of the zero set of $F$ is
	\begin{align}\label{eq_one_point_radial}
	\rho_1 
	= - \frac{1}{\pi}  \bigg(
	P'(0) + \frac{1}{4 P'(0)}
	\bigg).
	\end{align}
\end{coro}

We mention some applications of Theorem \ref{th_unsigned_H}; these are further developed in Section \ref{sec_app}.

In the context of Examples \ref{ex_true} and \ref{ex_poly} we obtain the following.

\begin{theorem}\label{th_poly}
The first intensity of the zero set of a true-type poly-entire function as in \eqref{eq_true} is $\frac{1}{\pi}  \big(q-\frac{1}{2} + \frac{1}{4q-2}\big)$, while that of a full-type one as in \eqref{eq_full} is $\frac{1}{2\pi}  \big(
q + \frac{1}{q} \big)$.
\end{theorem}
The base case $q=1$ is well-known as it corresponds to a Gaussian entire function (Example \ref{ex_gaf}), and follows from more general results \cite[Section 2]{gafbook}, while the case $q=2$ is implicit in \cite{MR2104882, MR3937291}, since, by \eqref{eq_wa}, it corresponds to the number of critical points of the weighted magnitude of a GEF. For large $q$ we see that a true-type poly-entire function has on average $\approx \tfrac{q}{\pi}$ zeros per unit area, while one of full-type has $\approx \tfrac{q}{2\pi}$ zeros per unit area.

As a second application, we consider the short-time Fourier transform (Example \ref{ex_stft}), and obtain the following.

\begin{theorem}[{First intensity of zeros of STFT of complex white noise}]
	\label{th_real_stft}
	Let $g\colon \mathbb{R} \to \mathbb{C}$ be a Schwartz function normalized
	by $\norm{g}_2=1$, and consider the following \emph{uncertainty constants}:
	\begin{align}\label{eq_uc}
	\begin{aligned}
	c_1 &:= \int_\RR t \,  \lvert g(t)\rvert^2 dt, 
	\qquad
	c_2 := \int_\RR t^2 \lvert g(t)\rvert^2 dt, 
	\qquad
	c_3 := \int_\RR \lvert g'(t)\rvert^2 dt,\\
	c_4 &:= -i\int_{\mathbb{R}} g(t) \overline{g'(t)} dt,
	\qquad
	c_5 :=
	\Im\bigg(\int_{\mathbb{R}} tg(t) \overline{g'(t)} dt\bigg).
	\end{aligned}
	\end{align}
	Then the zero set of $V_g\,\mathcal{N}$, i.e., the STFT of complex white noise with window $g$, has first intensity:
	\begin{align*}
	\rho_{1,g} \equiv 
	\frac{
		4 ( c_2 - c_1^2)c_3
		-4  c_2 c_4^2
		- 4 c_5^2 
		- 8 c_1 c_4 c_5 
		+1}{4  
		\sqrt{
			( c_2 - c_1^2)c_3
			-c_2 c_4^2
			-c_5^2 
			-2c_1 c_4 c_5
	}}.
	\end{align*}
	Concretely, for every Borel set $E \subseteq \mathbb{C}$:
	\begin{align*}
	\mathbb{E} \big[ \# \big\{z \in E: V_g \, \mathcal{N} (z)=0 \big\}\big] = \rho_{1,g} |E|.
	\end{align*}
	
	The constants $c_1, \ldots, c_5$ are real.
	When $g$ is real-valued, the expression for $\rho_{1,g}$ further simplifies because $c_4=c_5=0$.
\end{theorem}
If we interpret $|g(t)|^2$ as a probability density on $\mathbb{R}$, the uncertainty constants $c_1$ and $c_2$ correspond to the expected value and expected spread around the origin. The constants $c_3$ and $c_4$ have a similar meaning with respect to the Fourier transform of $g$. The constant $c_5$ is more subtle to interpret, as it involves correlations between $g$ and its Fourier transform. 

We spell out the particular case of Theorem \ref{th_real_stft} for Hermite windows:
\begin{align}\label{eq_intro_her}
h_{r}(t) = \frac{2^{1/4}}{\sqrt{r!}}\left(\frac{-1}{2\sqrt{\pi}}\right)^r
e^{\pi t^2} \frac{d^r}{dt^r}\left(e^{-2\pi t^2}\right), \qquad r \geq 0.
\end{align}
\begin{coro}\label{coro_intro_her}
	The expected number of zeros of the STFT of complex white noise with Hermite window $h_r$ \eqref{eq_intro_her} inside a Borel set $E \subset \mathbb{C}$ is
	\begin{align}\label{eq_x1}
	\mathbb{E} \big[ \# \big\{z \in E: V_{h_r} \, \mathcal{N} (z)=0 \big\}\big] = \Big(r + \tfrac{1}{2} + \tfrac{1}{4r+2} \Big) |E|.
	\end{align}
\end{coro}
Numerical simulations related to the zeros of the STFT of white noise with $h_0$ and $h_1$ as windows can be found in \cite[Chapter 15]{flandrin2018explorations} --- see also Figure \ref{fig:chargedh1}.

Note that the expression in Corollary \ref{coro_intro_her} is minimal for $r=0$ (Gaussian case). We will prove that this is in fact an instance of a general phenomenon.

\begin{theorem}[Uncertainty principle for the zeros of the STFT of white noise]\label{th_stft_up}
	Under the assumptions of Theorem \ref{th_real_stft},
	the minimal value of $\rho_{1,g}$ is $1$,
	and it is attained exactly when $g$ is a generalized Gaussian; that is,
	\begin{align}
	\label{eq_gaussian}
	g(t) =  \frac{\lambda}{\sqrt{\sigma}} e^{-\tfrac{\pi}{\sigma^2} \left[ (t-x_0 )^2
		+ i ( \xi_0 \cdot t + \xi_1 \cdot t^2)\right]
	}, 
	\qquad t \in \mathbb{R},
	\end{align}
	with $\sigma >0$, $\lambda \in \mathbb{C}$, $|\lambda|=2^{1/4}$,
	$x_0$, $\xi_0$, $\xi_1 \in \mathbb{R}$.
\end{theorem}
To compare, we note that, in terms of the uncertainty constants \eqref{eq_uc}, \emph{Heisenberg's uncertainty inequality} reads:
\begin{align}\label{eq_h}
(c_2-c_1^2) \cdot (c_3-c_4^2) \geq \frac{1}{4},
\end{align}
where $\norm{g}_2=1$, and is saturated by (translated and linearly modulated) Gaussian functions, see, e.g., \cite[Corollary 1.35]{folland89}. Generalized Gaussians \eqref{eq_gaussian} are sometimes called \emph{squeezed states} and minimize a refined version of \eqref{eq_h}
that involves the constant $c_5$ in \eqref{eq_uc}, known as the \emph{Robertson–Schr\"odinger} uncertainty relations \cite{trifonov2001schrodinger}. The proof of Theorem \ref{th_stft_up} exploits the invariance of squeezed states under the canonical transformations of the time-frequency plane (Weyl operators and metaplectic rotations \cite{folland89}). 

\subsection{Charged zeros}
We now look into weighting each zero $z$ of a GWHF $F$ with a \emph{charge} $\pm 1$, according to whether $F$ preserves or reverses orientation around $z$.
More precisely, we inspect the differential matrix $DF$
of $F$ considered as $F\colon \mathbb{R}^2 \to \mathbb{R}^2$ and define
\begin{align*}
\charge_z := \begin{cases}
1 & \mbox{if } \det DF(z) >0 \\
0 & \mbox{if } \det DF(z)  = 0 \\
-1 & \mbox{if } \det DF(z) < 0
\end{cases}.
\end{align*}
When zeros are interpreted as phase-singularities, charges correspond to the strength of their \emph{vorticity} \cite{berry2000phase}. Charges also appear naturally in the study of critical points of random functions, as one investigates the signature of corresponding Hessian matrices --- see \cite[Section 3.1]{MR2104882} for an extended discussion.

We encode charged zeros into the random measure:
\begin{align}\label{eq_charged_measure}
\mathcal{Z}^{\chsign}_F &:= \sum_{z \in \mathbb{C},\,
	F(z)=0} \charge_z \cdot \delta_z\,.
\end{align}
Our next result shows that the corresponding first intensity is independent of the twisted kernel $H$.
\begin{theorem}[First intensity of charged zeros]
\label{th_signed_H}
Let $F$ be a GWHF with twisted kernel $H$ satisfying the standing assumptions. Then the random signed measure $\mathcal{Z}^{\chsign}_F$ has first intensity $\rho_1^{\chsign} = \frac{1}{\pi}$, i.e.,
\begin{align*}
\mathbb{E} \bigg[ \sum_{z \in E,\, F(z)=0} \charge_z \bigg] = \frac{1}{\pi} |E|,
\qquad E \subseteq \mathbb{C} \mbox{ Borel set}.
\end{align*}
\end{theorem}
As mentioned after \eqref{eq_z1}, in the situation of Theorem \ref{th_signed_H} the zeros of $F$ are almost surely non-degenerate, and consequently $\charge_z=\pm 1$.
The fact that the intensity of charged zeros is constant is non-trivial, and shown in Section \ref{sec_first}.

Charges are straightforward to interpret in Examples \ref{ex_gaf}, \ref{ex_true}, and \ref{ex_poly}, as the transformation $G(z) \mapsto F(z)=e^{-|z|^2/2} G(z)$ preserves the sign of the Jacobian at a zero --- see Section \ref{sec_app_charges}. In the case of Gaussian entire functions, all zeros are positively charged due to the conformality of analytic functions, and, indeed, the first intensities prescribed by Theorems \ref{th_unsigned_H} and \ref{th_signed_H} coincide. On the other hand, Theorem \ref{th_poly} shows that a higher order poly-entire function has a large number of expected zeros per unit area, while, according to 
Theorem \ref{th_signed_H}, most of the corresponding charges cancel. Charges are, in expectation, in a certain equilibrium around the universal density $1/\pi$.

While zeros of first-order true poly-entire functions
correspond to critical points of weighted magnitudes of Gaussian entire functions (Example \ref{ex_true}), their charges summarize the signatures of the corresponding Hessian matrices --- see \cite[Section 3.1]{MR2104882} or Section \ref{sec_fd}. In fact, as we show in Section \ref{sec_fd}, Theorem \ref{th_signed_H} can be used to rederive a particular case of \cite[Corollary 5]{MR2104882}.

For the STFT of white noise (Example \ref{ex_stft}) the quantities related to charge are
\begin{align}\label{eq_newcharge}
\newcharge_{z} = \sgn \Big\{\Im \, \Big[  \partial_x (V_g\, \mathcal{N})(x,y) \cdot \overline{ \partial_y(V_g\, \mathcal{N})(x,y)} \,\Big] \Big\}, \qquad z=x+iy,
\end{align}
and Theorem \ref{th_signed_H} gives the following.
\begin{coro}[Equilibrium of charge for the STFT of complex white noise]\label{coro_charge_stft}
	Let $g\colon \mathbb{R} \to \mathbb{C}$ be a Schwartz function. Then the zeros of $V_g \, \mathcal{N}$ --- the STFT of complex white noise --- satisfy
	\begin{align*}
	\mathbb{E} \bigg[ 
	\sum_{z \in E,\, V_g\, \mathcal{N} (z) = 0} \newcharge_z
	\bigg] = |E|, \qquad E \subseteq \mathbb{C} \mbox{ Borel set}.
	\end{align*}
\end{coro}
Figure~\ref{fig:chargedh1} illustrates the distribution of charged zeros of one realization of $V_{h_1} \,\mathcal{N}$.
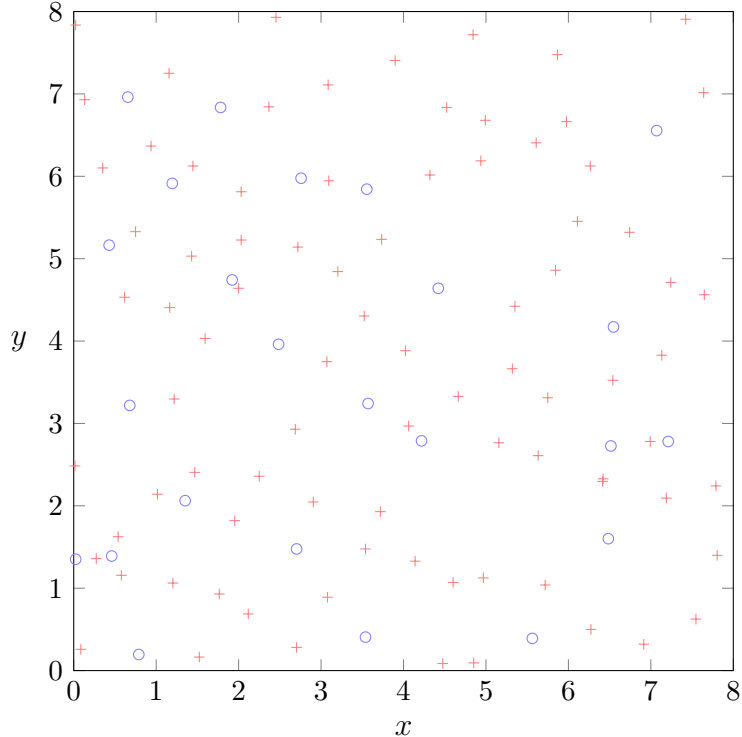
\begin{figure}[tbh]
\centering
\begin{tikzpicture}
\begin{axis}[
width=12cm,
xmin=0,xmax=8,
ymin=0,ymax=8,
%title=Convergence Plot,
xlabel={$x$},
ylabel={$y$},
ylabel style={rotate=-90},
tick label style={font=\small},
grid=none,
axis equal image 
]
\addplot+[red!50, only marks, mark=+] table[x expr=\thisrow{x}/128,y expr=\thisrow{y}/128] {./lmposadap.dat}; %positive charges
\addplot+[blue!50, only marks, mark=o] table[x expr=\thisrow{x}/128,y expr=\thisrow{y}/128] {./lmnegadap.dat}; %positive charges
\end{axis}
\end{tikzpicture}
\vspace{-3mm}
\caption{Distribution of the charged zeros of one realization of $V_{h_1} \mathcal{N}$ on $[0,8]\times [0,8]$.
Red plus signs correspond to positive charges while blue circles correspond to negative charges.
The number of zeros is $109$ and thus close to the expectation $64\cdot 5/3 \approx 106.7$. The total charge of $61$ is also close to the expectation of $64$.}
\label{fig:chargedh1}
\end{figure}

\subsection{Fluctuation of aggregated charge}
While a general GWHF can have many expected zeros (Theorem \ref{th_unsigned_H}), the corresponding expected charges almost balance out (Theorem \ref{th_signed_H}), adding up to the universal density $1/\pi$. We now look into the stochastic fluctuation of charge when aggregated inside large observation sets, and the extent to which equilibrium is observed at large scales.

A point process is called \emph{hyperuniform} if the variance of the number of particles within an observation disk of radius $R$ is asymptotically smaller than the corresponding expected number of points \cite{torquato2016hyperuniformity}. Such fluctuations are also called \emph{non-extensive} \cite{ghosh2017fluctuations}, and are anomalously small in comparison to those in ordinary fluids and amorphous solids. Originally introduced in material science, hyperuniformity provides a unified framework to classify crystals and quasicrystals. The notion was subsequently developed into an abstract statistical notion, and found applications in a broad range of topics in physics, number theory, and biology \cite{MR3815253}. In particular, hyperuniformity can be formulated, even quantitatively, for charged point processes such as \eqref{eq_charged_measure}, and certain classical results can be recast in this light. For example, fluctuations of charged Coulomb systems within observation disks or radius $R$, if non-extensive, are known to be dominated by the observation perimeter $O(R)$ \cite{le83, martin1980charge}.

Our last result shows that the fluctuations of the aggregated charge of zeros of GWHF with radial twisted kernels are non-extensive, and moreover provides an asymptotic expression for the variance of charge.

\begin{theorem}[Hyperuniformity of charge]
	\label{th_signed_variance}
	Let $F$ be a GWHF with twisted kernel $H$ satisfying the standing assumptions. Assume further that
	$H(z)=P\big(\lvert z\rvert^2\big)$,
	where $P\colon \mathbb{R} \to \mathbb{R}$
	is $C^2(\mathbb{R})$
	and
	\begin{equation*}
	\sup_{r \geq 0}
	\big(
	\abs{P(r)} + \abs{P'(r)} + \abs{P''(r)}
	\big) r^2 < \infty.
	\end{equation*}
	Then the charged measure of zeros $\mathcal{Z}^{\chsign}_F$ satisfies the following: there exists a constant $C=C_H>0$ such that
	for all $z \in \mathbb{C}$,
	\begin{align*}
	\mathrm{Var}\big[\mathcal{Z}^{\chsign}_F(B_R(z))\big]
	\leq C R, \qquad R>0,
	\end{align*}
	while 
	\begin{align*}
	\tfrac{1}{R} \mathrm{Var}
	\Big[
	\mathcal{Z}^{\chsign}_F(B_R(z))
	\Big]
	\to  
	\frac{1}{\pi}  
	\int_0^\infty 
	\frac{2 r^2 P'(r^2)^2}{1-P(r^2)^2}
	dr,
	\qquad \mbox{ as } {R \to \infty},
	\end{align*}
	uniformly on $z$.
\end{theorem}

The hypothesis of Theorem \ref{th_signed_variance} is satisfied in Example \ref{ex_gaf} (Gaussian entire functions, where all charges are positive and more refined results exist \cite[Section 3.5]{gafbook}), and in poly-entire contexts (Examples \ref{ex_true}, \ref{ex_poly}), as well as for the short-time Fourier transform of white noise (Example \ref{ex_stft}) with Hermite windows --- see Section \ref{sec_app_charges}. The case of order one pure poly-entire functions may be interesting in relation to the classification of critical points of Gaussian entire functions --- see Section \ref{sec_fd}.

In the context of Theorem \ref{th_signed_variance}, whenever the one-point function of the zero set of $F$ is large, \emph{most of the positively charged zeros tend to be surrounded by negatively charged ones}, a phenomenon that in the stationary setting is called (almost perfect) \emph{screening} \cite{MR662182,wilkinson2004screening}. In the
twisted stationary setting, this phenomenon is \emph{universally valid}, independently of the particular kernel $H$. 
The significance of hyperuniformity thus concerns the empirical observability of the ensemble average claimed in Theorem \ref{th_signed_H}: \emph{universal screening is observed with growing probability at all sufficiently large scales.} For example, by Markov's inequality, Theorem \ref{th_signed_variance} implies that
\begin{align*}
	\mathbb{P} \Big[
	\big|
	\tfrac{1}{\abs{B_R(z)}} \mathcal{Z}^{\chsign}_F(B_R(z)) - \tfrac{1}{\pi} \big| > \varepsilon 	\Big] = O_\varepsilon \Big(\frac{1}{R^3}\Big),
\end{align*}
which is consistent with the experiment in Figure~\ref{fig:chargedh1}, where the prescribed equilibrium is observable already in one realization of a GWHF.

\subsection{Organization}
Our main tool is direct computation with Kac-Rice formulae and exploitation of the invariance relation \eqref{eq_twisted_stationairy}. 
Section \ref{sec_notation} introduces background results and required adaptations to our setting. 
In Section \ref{sec_first}, we prove all results related to first intensities (Theorem \ref{th_unsigned_H}, Theorem \ref{th_signed_H} and Corollary \ref{coro_one_point_radial}). 
In Section \ref{sec_var}, we
study second order statistics of charged zeros and prove Theorem \ref{th_signed_variance}. 
Section \ref{sec_app} develops applications to Examples \ref{ex_gaf}, \ref{ex_true}, \ref{ex_poly}, and \ref{ex_stft}, including proofs of Theorem \ref{th_poly}, Theorem \ref{th_real_stft}, Corollary \ref{coro_intro_her}, Theorem \ref{th_stft_up}, and Corollary \ref{coro_charge_stft}. 
The short-time Fourier transform plays a prominent role, as time-frequency techniques are also brought to bear on the other examples. Section \ref{sec_aux} contains auxiliary results, including a lengthy calculation, for which we also provide a Python worksheet at \url{https://github.com/gkoliander/gwhf}. Section \ref{sec_conc} offers conclusions, a discussion on open problems, and perspectives on future work.

\section{Preliminaries}\label{sec_notation}
\subsection{Notation}
We use $t$ for real variables and $z,w$ for complex variables. We always use the notation $z=x+iy$, $w=u+iv$, with $x,y,u,v \in \mathbb{R}$. 
The real and imaginary parts of $z \in \mathbb{C}$ are otherwise denoted $\Re(z)$ and $\Im(z)$, respectively. 
The differential of the (Lebesgue) area measure on the plane will be denoted for short $dA$, while the measure of a set $E$ is $|E|$.
The derivatives of a function $F\colon \mathbb{C} \to \mathbb{C}$ interpreted as $F\colon \mathbb{R}^2 \to \mathbb{R}^2$ are denoted by $F^{(1,0)}$ (real coordinate) and $F^{(0,1)}$ (imaginary coordinate). 
Higher derivatives are denoted by $F^{(k,\ell)}$.

Vectors $\vecc \in \mathbb{C}^n$ are identified with column matrices
$\vecc \in \mathbb{C}^{n \times 1}$; $\vecc^t$ denotes transposition, while $\vecc^*$ denotes transposition followed by coordinatewise conjugation.
We let
\begin{align*}
J=\begin{bmatrix}
0 & 1 \\ -1 & 0
\end{bmatrix}
\end{align*}
denote the matrix with the property:
\begin{align}\label{eq_symp}
(z, w)^* J (z,w) = -2i \Im(z\bar w), \qquad (z,w) \in \mathbb{C}^2.
\end{align}

The \emph{Jacobian} of $F\colon \mathbb{C} \to \mathbb{C}$ at $z \in \mathbb{C}$ is the determinant of its differential matrix $DF$ considered as $F\colon \mathbb{R}^2 \to \mathbb{R}^2$:
\begin{align*}
\jac F(z) := \det DF(z).
\end{align*}
The following observations will be used repeatedly:
\begin{align}\label{eq_Jim}
\jac F(z) = \det DF(z)
= - \Im\big[  F^{(1,0)}(z) \cdot \overline{ F^{(0,1)}(z)}\big].
\end{align}

\subsection{Gaussian vectors and intensities}

By a  Gaussian vector we always mean a circularly symmetric complex Gaussian random vector, i.e., a random vector $X$ on $\mathbb{C}^n$ such that
$(\Re (X), \Im (X))$ is normally distributed, has zero mean, and vanishing \emph{pseudo-covariance}:
\begin{align*}
\mathbb{E}\big[ X   X^t \big] = 0.
\end{align*}
A complex Gaussian vector $X$ on $\mathbb{C}^n$ is thus determined by its \emph{covariance matrix}
\begin{align*}
\cov [X] = \mathbb{E} \big[ X   X^* \big].
\end{align*}
If $\cov [X]$ is non-singular, then $X$ is absolutely continuous and has probability density
\begin{align}\label{eq_complex_gaussian}
f_X \vecc = \frac{1}{\pi^n\det{\cov [X]}}
\exp\big({-\vecc^* \,(\cov [X])^{-1} \, \vecc}\big).
\end{align}
Gaussian vectors are not a priori assumed to have non-singular covariances. The zero vector, for example, is a singular Gaussian vector.

If $(X,Y)$ is a Gaussian vector on $\mathbb{C}^{n+m}$ and $h\colon \mathbb{C}^{n} \to \mathbb{R}$ is a function, the conditional expectation $\mathbb{E}\big[h(X)\,\big\vert\,Y=0\big]$ is defined by \emph{Gaussian regression}. Informally,
this involves finding a linear combination of $X,Y$ which is uncorrelated to $Y$. The following remark makes this intuition precise.

\begin{rem}[\emph{Gaussian regression}]\label{lemma_condition}
	Let $(X,Y)$ be a circularly symmetric Gaussian random vector in $\mathbb{C}^{n+m}$
	with a (possibly singular) covariance matrix
	\begin{align*}
		\cov[(X,Y)] = \begin{bmatrix}
			A & B \\ B^* & C
		\end{bmatrix},
	\end{align*}
	where $A=\cov[X] \in \mathbb{C}^{n \times n}$, $B \in \mathbb{C}^{n \times m}$,
	and $C=\cov[Y] \in \mathbb{C}^{m \times m}$. 
	Assume further that $C$ is nonsingular.
	Let $Z$ be a circularly symmetric Gaussian random vector in $\mathbb{C}^{n}$ with covariance
	\begin{align*}
		\cov[Z] = A-B C^{-1} B^*.
	\end{align*}
	Then, for any locally bounded $h \colon \mathbb{C}^{n} \to \mathbb{R}$
	\begin{align*}
		\mathbb{E} \big[ h(X) \, \big\vert\, Y=0 \big] = \mathbb{E} \big[ h(Z)\big]
	\end{align*}
	is the Gaussian regression version of the conditional expectation \cite[eq.~(1.5)]{level}.
\end{rem}

Whenever it exists, the \emph{first intensity} or \emph{one-point intensity} of a random signed measure $\mu$ on $\mathbb{C}$ is a measurable function $\rho\colon \mathbb{C} \to \mathbb{R}$ such that
\begin{align*}
\mathbb{E} \big[
\mu(E) \big] = \int_E \rho(z) \,dA(z), \qquad E\mbox{ Borel set}.
\end{align*}
Second order intensities are defined in the article as needed.
Objects related to charged zeros are denoted with a superscript $\chsign$.

Background on random Gaussian functions can be found in \cite{adler,level}.

\subsection{Kac-Rice formulae}
The formulae that describe the statistics of the level sets of Gaussian functions are generically known as Kac-Rice formulae. The following result is quoted from \cite[Theorem 6.2]{level} --- with the notation $\jac f := \det Df$, for $f\colon \mathbb{R}^d\to\mathbb{R}^d$.
\begin{prop}[Expected number of roots]\label{prop_kr1}
Let $U \subset \mathbb{R}^d$ be open, $Z\colon U \to \mathbb{R}^d$ a Gaussian random field, and $u \in \mathbb{R}^d$. Assume that:
\begin{itemize}
\item[(i)] Almost surely, the function $t \mapsto Z(t)$ is $C^1$,
\item[(ii)] For each $t \in U$, $Z(t)$ has a non-degenerate distribution (i.e., its covariance it positive-definite),
\item[(iii)] $\mathbb{P} \left\{ \text{There exists } t \in U \text{ such that } Z(t)=u \text{ and } \jac Z(t) =0 \right\} =0$.
\end{itemize}
Then for every Borel set $E \subset U$:
\begin{align}\label{eq_1}
\mathbb{E} \big[ 
\# \{t\in E\,:\, Z(t)=u \} \big] = \int_E \mathbb{E} \big[
\left| \jac Z(t) \right|\, \big|\, Z(t)=u \big] p_{Z(t)}(u) \,  dt,
\end{align}
where $p_{Z(t)}$ is the probability density function of $Z(t)$. In addition, 
both sides of \eqref{eq_1} are finite if $E$ is compact.
\end{prop}
The following weighted version of Proposition \ref{prop_kr1} is a particular case of \cite[Theorem 6.4]{level}.
\begin{prop}[Expected number of weighted roots]\label{prop_kr2}
Under the assumptions of Proposition~\ref{prop_kr1}, let $\varphi\colon \mathbb{R}^d \times \mathbb{R}^{d\times d} \to \mathbb{R}$ be bounded and continuous and $E \subset U$ compact. Then
\begin{align*}
\mathbb{E} \left[ 
\sum_{t \in E, Z(t)=u} \varphi\big(Z(t), DZ(t) \big)
\right] = \int_E \mathbb{E} \big[
	\left| \jac Z(t) \right| \varphi\big(Z(t), DZ(t) \big) \, \big|\, Z(t)=u \big] p_{Z(t)}(u) \,  dt.
\end{align*}
\end{prop}

\section{Preparatory Calculations with GWHF}
\subsection{Covariance structure of first derivatives}
\label{sec:covstruct}
As a first step in the investigation of a GWHF $F$, we describe the stochastics of the Gaussian vector
\begin{align}\label{eq_FF}
\big(F(z), F^{(1,0)}(z), F^{(0,1)}(z)\big)
\end{align}
at a given point $z\in \mathbb{C}$. We start by calculating the covariances
\begin{align*}
&\mathbb{E}
\left[
F(z) \cdot \overline{F(w)}
\right]=
H(z-w) e^{i (yu-xv)},
\\
&\mathbb{E}
\left[
F(z) \cdot \overline{F^{(1,0)}(w)}
\right]
= \left(-H^{(1,0)}(z-w) + iy H(z-w) 
\right)
e^{i (yu-xv)},
\\
&\mathbb{E}
\left[
F(z) \cdot \overline{F^{(0,1)}(w)}
\right]
= \left(
-H^{(0,1)}(z-w) - ix H(z-w) 
\right)
e^{i (yu-xv)},
\\
&\mathbb{E}
\left[
F^{(1,0)}(z) \cdot \overline{F^{(1,0)}(w)}
\right]
\\
&\quad
= \left(-H^{(2,0)}(z-w)+iy H^{(1,0)}(z-w) 
+ 
iv H^{(1,0)}(z-w)+yv H(z-w)
\right)
e^{i (yu-xv)},
\\
&\mathbb{E}
\left[
F^{(1,0)}(z) \cdot \overline{F^{(0,1)}(w)}
\right]
\\
&\quad
= \big(
- H^{(1,1)}(z-w)-iH(z-w)
- ixH^{(1,0)}(z-w)
+ iv H^{(0,1)}(z-w)
- xvH(z-w)
\big)
e^{i (yu-xv)},
\\
&\mathbb{E}
\left[
F^{(0,1)}(z) \cdot \overline{F^{(0,1)}(w)}
\right]
\\
& \quad= \left(
-H^{(0,2)}(z-w)-ix H^{(0,1)}(z-w)
-iu H^{(0,1)}(z-w) + xuH(z-w)
\right)
e^{i (yu-xv)}.
\end{align*}
The following lemma will help us simplify further calculations.
\begin{lemma}
\label{lemma_HH}
Let $H$ be a twisted kernel satisfying the standing assumptions. Then $H(0)=1$,
\begin{align*}
&H^{(1,0)}(0), H^{(0,1)}(0) \in i \mathbb{R},
\\
&H^{(2,0)}(0), 
 H^{(0,2)}(0),
 H^{(1,1)}(0) \in \mathbb{R}.
\end{align*}
\end{lemma}
\begin{proof}
The conclusion follows directly from 
\eqref{eq_pd}.
\end{proof}
We now specialize the previous calculations at $z=w$ and see that the covariance matrix of \eqref{eq_FF} is
\begin{align}
  \Gamma(z) 
  & =
  \begin{pmatrix}
    1
    &  
     iy 
    & 
     - ix 
    \\
    - iy 
    & 
    y^2 
    & 
    -i- xy
    \\
     ix 
    &
    i- xy
    &
    x^2
  \end{pmatrix}
  +
  H^{(1,0)}(0)
  \begin{pmatrix}
    0
    &  
    -1
    & 
    0  
    \\
    1 
    & 
    2iy
    & 
    - ix
    \\
    0
    &
    - ix
    &
    0
  \end{pmatrix}
  \notag
  \\*
  & \quad +  
  H^{(0,1)}(0)
  \begin{pmatrix}
    0
    &  
    0
    & 
    -1  
    \\
    0
    & 
    0
    & 
    iy 
    \\
    1
    &
    iy
    &
    -2ix 
  \end{pmatrix}
  + 
  \begin{pmatrix}
    0
    &  
    0
    & 
    0  
    \\
    0
    & 
    -H^{(2,0)}(0)
    & 
    - H^{(1,1)}(0)
    \\
    0
    &
    - H^{(1,1)}(0)
    &
    -H^{(0,2)}(0)
  \end{pmatrix}.
  \label{eq:covfz}
\end{align}
We will be mainly interested in conditional expectations of the form
\begin{align}\label{eq_E}
\mathbb{E}\big[ h\big(F^{(1,0)}(z), F^{(0,1)}(z) \big) \, \big\vert \, F(z)=0\big].
\end{align}
According to Remark~\ref{lemma_condition}, these are $\mathbb{E}\big[ h(Z) \big]$ where $Z\in \mathbb{C}^2$ is a circularly symmetric Gaussian random vector with covariance matrix:
\begin{align}
\label{eq_Omega1}
\Omega = \begin{bmatrix}
\alpha & \gamma \\
\bar\gamma & \beta
\end{bmatrix}
=
\begin{bmatrix}
 -H^{(2,0)}(0)  + \big(H^{(1,0)}(0)\big)^2 
& - H^{(1,1)}(0)-i 
+H^{(1,0)}(0)H^{(0,1)}(0)
\\
 - H^{(1,1)}(0)+i +H^{(1,0)}(0)H^{(0,1)}(0)
& -H^{(0,2)}(0) + \big( H^{(0,1)}(0) \big)^2 
\end{bmatrix}.
\end{align} 
As we can see, the covariance matrix \eqref{eq_Omega1} and thus the expectation \eqref{eq_E} do not depend on the specific given point $z$. The following result formalizes these observations. 

\begin{prop}
	\label{prop_reg}
	Let $F$ be a GWHF satisfying the standing assumptions. Then:
  \begin{enumerate}
  \renewcommand{\theenumi}{(\alph{enumi})}
  \renewcommand{\labelenumi}{(\alph{enumi})}
		\item \label{en:absfstat}The field of absolute values $\abs{F}$ 
    is stationary, i.e., defining $G_w(z):= F(z-w)$ for each $z\in \mathbb{C}$ and any $w \in \mathbb{C}$, then the random functions $\abs{G_w}$ and $\abs{F}$ have the same distribution.
    \item \label{en:jacfcondstat}For any 
    locally bounded function $h\colon \mathbb{R} \to \mathbb{R}$, the conditional expectation
    \begin{equation}
      \mathbb{E}\big[h\big(\jac F(z)\big) \, \big\vert\, F(z)=0 \big]
      = \mathbb{E}\big[h \big({-}\Im[  Z_1  \overline{ Z_2}]\big) \big]
    \notag
    \end{equation}
    where $Z = (Z_1\, Z_2)\in \mathbb{C}^2$ is a circularly symmetric Gaussian random variable with covariance matrix given by $\Omega$ in \eqref{eq_Omega1}.
    In particular, $\mathbb{E}\big[h\big(\jac F(z)\big) \, \big\vert\, F(z)=0 \big]$
    does not depend on the choice of  $z\in \mathbb{C}$.
		\item \label{en:zerossimp}The zeros of $F$ are almost surely simple and thus isolated. More precisely,
		\begin{align*}
		\mathbb{P} \left\{ \text{There exists } z \in \mathbb{C} \text{ such that } F(z)=0 \text{ and } \jac F(z) =0 \right\} =0. 
		\end{align*}
	\end{enumerate}
\end{prop}
\begin{proof}

	For \ref{en:absfstat} note that the twisted stationarity condition \eqref{eq_twisted_stationairy}
	implies that $e^{i\Im(\cdot \overline{w})} F(\cdot-w)$ and $F(\cdot)$ are circularly symmetric complex Gaussian fields with the same covariance structure. 
  The claim about absolute values follows immediately since $\lvert e^{i\Im(z \overline{w})} F(z-w)\rvert = \abs{G_w(z)}$. 
  
  Part \ref{en:jacfcondstat} follows from  discussion above, as $\jac F(z)=- \Im\big[  F^{(1,0)}(z) \, \overline{ F^{(0,1)}(z)}\big]$.
	
	For \ref{en:zerossimp}, we apply \cite[Proposition 6.5]{level}.
  The required hypotheses are that $F$ be $C^2$ almost surely, as we assume, and that the probability density of $F(z)$ be bounded near $0$, uniformly in $z$. 
  This last requirement is satisfied because $F(z)$ is a circularly symmetric complex Gaussian vector with variance
	$\var[F(z)] = H(0)=1$.
\end{proof}

\subsection{Kac-Rice formulae for GWHF}
The second preparatory step is to check that various Kac-Rice formulae are applicable to GWHF and obtain corresponding intensities for zeros.
\begin{lemma}\label{lemma_kr}
Let $F$ be a GWHF with twisted kernel $H$ satisfying the standing assumptions.
Then the first intensities of the uncharged and charged zero sets are independent of $z$ and given by
\begin{align}
\label{eq_KR}
\rho_1 
&= \tfrac{1}{\pi}
\mathbb{E} \left[
\abs{\jac F(z)} \,\big|\, F(z)=0 \right], \qquad \mbox{for any } z \in \mathbb{C},
\\
\label{eq_KRg}
\rho^{\chsign}_1 
&= \tfrac{1}{\pi}
\mathbb{E}\left[
\jac F(z) \,\big|\, F(z)=0 \right], \qquad \mbox{for any } z \in \mathbb{C}.
\end{align}
That is, the random measures \eqref{eq_z1} and \eqref{eq_charged_measure} satisfy
\begin{align*}
\mathbb{E} \big[ \mathcal{Z}_F(E)\big] &= \rho_1 |E|, \\
 \mathbb{E} \big[ \mathcal{Z}^{\chsign}_F(E)\big] &= \rho^{\chsign}_1  |E|,
\end{align*}
for every Borel set $E \subseteq \mathbb{C}$.

In addition, we define the \emph{semi-charged two-point intensity} $\tau_2^{\chsign}\colon \mathbb{C} \to \mathbb{R}$ by
\begin{align}
\label{eq_KR2}
\tau_2^{\chsign}(z-w) 
& = 
\frac{
\mathbb{E}\big[
\jac F(z) \, \jac F(w) \,\big|\, F(z)=F(w)=0 \big]
}{
\pi^{2}\big(1 - |{H(z-w)}|^2\big)
}, \qquad \mbox{for any } z,w \in \mathbb{C}.
\end{align}
Then $\tau_2^{\chsign}$ is well-defined and serves as density for the following semi-charged factorial moment:
	\begin{align}
  \label{eq:tauint}
	\mathbb{E}
	\big[
	\big(\mathcal{Z}^{\chsign}_F (E)\big)^2-\mathcal{Z}_F (E)\big]
	&= \int_{E\times E} \tau_2^{\chsign}(z-w) \, dA(z) dA(w), \qquad
	E \subseteq \mathbb{C} \mbox{ Borel set}.
	\end{align}
\end{lemma}
\begin{proof}
	We first apply Proposition~\ref{prop_kr1} to obtain:
	\begin{align*}
	\rho_1(z) &= \mathbb{E} \left[
	\abs{\jac F(z)} \,\big|\, F(z)=0 \right]  p_{F(z)} (0).
	\end{align*}
	The required regularity hypotheses are verified by Proposition~\ref{prop_reg}. 
  Since $F(z)$ is a Gaussian circularly symmetric complex random variable with zero mean and variance $H(0)=1$, $p_{F(z)}(0) = \tfrac{1}{\pi}$, and \eqref{eq_KR} follows. 
  The independence of $z$ follows from Property~\ref{en:jacfcondstat} in Proposition~\ref{prop_reg}.
  
  Similarly, Proposition~\ref{prop_kr2} gives:
	\begin{align}
	\label{eq_KRg2}
	\mathbb{E}\Bigg[
	\sum_{z \in E, F(z)=0} \func(\jac F(z))
	\Bigg] &= 
  \tfrac{1}{{\pi}} 
  \int_E
	\mathbb{E}\big[
	\abs{\jac F(z)} \,
	\func(\jac F(z))
	\, \big| \, F(z)=0 \big] 
	dA(z),
	\end{align}
	for all compact $E \subseteq \mathbb{C}$ and bounded and continuous $\func \colon \mathbb{R} \to \mathbb{R}$. 
  Formally applying this formula to the non-continuous function $\func(x)=\sgn(x)$ yields \eqref{eq_KRg}. 
  To justify such application, fix a compact set $E\subseteq \mathbb{C}$ and
	let $\func_n \colon \mathbb{R} \to [-1,1]$
	be continuous and such that $\func_n(x)=\sgn(x)$, for $\abs{x} > 1/n$.
	First note that 
	\begin{align}\label{eq_abc}
	\sum_{z \in E, F(z)=0} \func_n(\jac F(z))
	\longrightarrow \sum_{z \in E, F(z)=0} \sgn(\jac F(z))
	\end{align}
	almost surely, as convergence can only fail when $F(z)=0$ and $\jac F(z)=0$, and this is a zero probability event according to Property~\ref{en:zerossimp} in Proposition \ref{prop_reg}. 
  To show that \eqref{eq_abc}
	also holds in expectation, we estimate
	\[
  \Bigg\lvert\sum_{z \in E, F(z)=0} \func_n(\jac F(z)) \Bigg\rvert  
  \leq \# \{z \in E, F(z)=0\},
  \] 
  note that
	$\mathbb{E} \big[ \# \{z \in E, F(z)=0\} \big] = \rho_1 \abs{E} <\infty$,
	and invoke the Dominated Convergence Theorem.
  We now inspect the right-hand side of \eqref{eq_KRg2}. 
  By Property~\ref{en:jacfcondstat} in Proposition~\ref{prop_reg},
  \begin{equation}
  \label{eq:condabsfuncn}
    \mathbb{E}\big[
    \abs{\jac F(z)} \,
    \func_n(\jac F(z))
    \, \big| \, F(z)=0 \big]
    = \mathbb{E}\big[
    \abs{- \Im[  Z_1  \overline{ Z_2}]} \,
    \func_n(- \Im[  Z_1  \overline{ Z_2}])
    \big]
  \end{equation}
  and 
  \begin{equation}
  \label{eq:condjacwoabs}
    \mathbb{E}\big[
    \jac F(z)
    \, \big| \, F(z)=0 \big]
    = \mathbb{E}\big[
    \Im[  Z_1  \overline{ Z_2}]
    \big]
  \end{equation}
  where $Z = (Z_1\, Z_2)\in \mathbb{C}^2$ is a circularly symmetric Gaussian random variable with covariance matrix $\Omega$ given by \eqref{eq_Omega1}.
  Since 
  $\abs{- \Im[  Z_1  \overline{ Z_2}]} \,
    \func_n(- \Im[  Z_1  \overline{ Z_2}]) 
    \to  Z_1  \overline{ Z_2}$
  almost surely, 
  and \[\abs{- \Im[  Z_1  \overline{ Z_2}] \,
  \func_n(- \Im[  Z_1  \overline{ Z_2}])}
  \leq
  \abs{- \Im[  Z_1  \overline{ Z_2}] },\]
  while
  \begin{equation}
    \mathbb{E}\big[
    \abs{- \Im[  Z_1  \overline{ Z_2}] }
    \big]
    = \mathbb{E}\big[
	\abs{\jac F(z)} \, \big\vert\, F(z)=0 \big] 
    = \pi \rho_1 \abs{E} 
    < \infty\,,
  \end{equation}
  we can again invoke the Dominated Convergence Theorem to conclude that 
  the right-hand side in \eqref{eq:condabsfuncn} converges to the right-hand side of \eqref{eq:condjacwoabs}. 
  Summarizing, we have that
  \begin{align}
  \mathbb{E}\Bigg[
	\sum_{z \in E, \,F(z)=0} \sgn(\jac F(z))
	\Bigg]
  &= 
	\lim_{n\to \infty}\mathbb{E}\Bigg[
	\sum_{z \in E, \,F(z)=0} \func_n(\jac F(z))
	\Bigg] 
  \notag \\
  & = 
  \lim_{n\to \infty}
  \tfrac{1}{{\pi}} 
  \int_E
	\mathbb{E}\big[
	\abs{\jac F(z)} \,
	\func_n(\jac F(z))
	\, \big\vert \, F(z)=0 \big] 
	dA(z)
  \notag \\
  & =
  \tfrac{1}{{\pi}} 
  \int_E
	\mathbb{E}\big[
	\jac F(z)
	\, \big\vert \, F(z)=0 \big] 
	dA(z),
  \end{align}
  which yields \eqref{eq_KRg}.
  	For \eqref{eq_KR2}, we first note that
	\begin{align*}
	\big(\mathcal{Z}^{\chsign}_F (E)\big)^2-\mathcal{Z}_F (E)
	=
	\sum_{\substack{z,w \in E, z \not= w \\ F(z)=F(w)=0}}
	\mathrm{sgn} (\jac F(z)) \cdot \mathrm{sgn} (\jac F(w)).
	\end{align*}
	Let $\delta>0$ and consider the random Gaussian field
	on $\mathbb{C}^2$
  given by
	\begin{align*}
	\tilde{F}(z,w) = (F(z),F(w)).
	\end{align*}
	We apply Proposition~\ref{prop_kr2} to $\tilde{F}$ and use a regularization argument as before to learn that, for any Borel set $\tilde E \subseteq \mathbb{C}^2$,
	\begin{align}\label{eq_lllll}
	\begin{aligned}
	&\mathbb{E} \bigg[
	\sum_{(z,w) \in \tilde E, \abs{z-w} \geq \delta}
	\mathrm{sgn} (\jac F(z)) \cdot \mathrm{sgn} (\jac F(w))
	\bigg]
	\\
	&\qquad=
	\lambda
	\int_{\tilde E \setminus \{\abs{z-w} < \delta\}}
	\mathbb{E}\big[
	\jac F(z) \cdot \jac F(w) | F(z)=F(w)=0
	\big],
	\end{aligned}
	\end{align}
	where
	\begin{equation}\label{eq_llll}
	\lambda=\pi^{-2}\Big(1 - \abs{H(z-w)}^2\big)^{-1}.
	\end{equation}
	To apply the weighted Kac-Rice formula it is important
	that the covariance matrix of $\tilde{F}(z,w)$
	be non-singular, as granted by \eqref{eq_non_deg}.
	Proposition~\ref{prop_kr2} thus gives \eqref{eq_lllll},
	where $\lambda$ is the value of the probability density of $\tilde{F}(z,w)$ at $0$, which is indeed given by \eqref{eq_llll}. 
  We let $E \subseteq \mathbb{C}$ be compact, choose $\tilde E \subseteq E \times E$ in \eqref{eq_lllll} according to the signs of $\jac F(z)$ and $\jac F(w)$, let $\delta \to 0$, and use the monotone convergence theorem to deduce
	\eqref{eq:tauint}, albeit with a function
	depending on $(z,w)$ in lieu of $\tau_2^{\chsign}$.	
It thus remains to show that $\mathbb{E}\big[
\jac F(z) \, \jac F(w) \,\big|\, F(z)=F(w)=0 \big]$ depends only on the difference $z-w$. To this end, note that the twisted stationarity condition \eqref{eq_twisted_stationairy}
	means that 
	\begin{align*}
F_\zeta(z) := e^{i\Im(z \overline{\zeta})} F(z-\zeta), \qquad z \in \mathbb{C},
	\end{align*}
	and $F$ are circularly symmetric complex Gaussian fields with the same covariance. Thus, 
  \begin{align}\label{eq_eeww2}
	\mathbb{E} \Big[
	\jac F(z) \cdot \jac F(w) \,\big|\, F(z)=F(w)=0
	\Big] = \mathbb{E} \Big[
	\jac F_\zeta(z) \cdot \jac F_\zeta(w) \, \big| \, F_\zeta(z)=F_\zeta(w)=0
	\Big].
	\end{align}
  Let us calculate the right-hand side of the  previous equation. Writing $\zeta=a+ib$, we compute
	\begin{align*}
	F^{(1,0)}_\zeta(z) = e^{i\Im(z \overline{\zeta})} \big[ F^{(1,0)}(z-\zeta) - i b F(z-\zeta) \big],
	\\
	F^{(0,1)}_\zeta(z) = e^{i\Im(z \overline{\zeta})} \big[ F^{(0,1)}(z-\zeta) + i a F(z-\zeta) \big].
	\end{align*}
	Similar equations hold of course for $w$ in lieu of $z$.
	We note that the event 
	\begin{align}\label{eq_event}
	\{F_\zeta(z)=F_\zeta(w)=0\}
	\end{align}
	is precisely the event $\{F(z-\zeta)=F(w-\zeta)=0\}$, and that under this event,
	\begin{align*}
	F^{(1,0)}_\zeta(z) = e^{i\Im(z \overline{\zeta})} F^{(1,0)}(z-\zeta),
	\\
	F^{(0,1)}_\zeta(z) = e^{i\Im(z \overline{\zeta})} F^{(0,1)}(z-\zeta),
	\end{align*}
	and similarly for $w$ in lieu of $z$. As a consequence, under \eqref{eq_event},
	$\jac F_\zeta(z)=\jac F(z-\zeta)$ and 
	$\jac F_\zeta(w)=\jac F(w-\zeta)$.
	 Plugging these observations into \eqref{eq_eeww2} we conclude that
	\begin{align*}
	&\mathbb{E} \big[
	\jac F(z) \cdot \jac F(w) \,\big\vert\, F(z)=F(w)=0
	\big]
	\\
	&\qquad= \mathbb{E} \big[
	\jac F(z-\zeta) \cdot \jac F(w-\zeta) \,\big\vert\, F(z-\zeta)=F(w-\zeta)=0
	\big].
	\end{align*}
	That is, the number $\mathbb{E} \big[
	\jac F(z) \cdot \jac F(w) \,\big\vert\, F(z)=F(w)=0
	\big]$ depends only on $z-w$.
\end{proof}

\section{First Intensities}\label{sec_first}
We can now derive the main results on first intensities.

\begin{proof}[Proof of Theorem \ref{th_unsigned_H}]
By Kac-Rice's formula \eqref{eq_KR},
\begin{align*}
\rho_1 
&= \frac{1}{\pi}
\mathbb{E} \big[
\abs{\jac F(z)} \,\big|\, F(z)=0 \big]
 = \frac{1}{{\pi}} \mathbb{E} \big[
\bigabs{- \Im\big[  F^{(1,0)}(z) \, \overline{ F^{(0,1)}(z)}\big]} \,\big\vert\, F(z)=0 \big].
\end{align*}
By Property~\ref{en:jacfcondstat} in Proposition~\ref{prop_reg},
  \begin{equation*}
    \mathbb{E} \big[
\bigabs{- \Im\big[  F^{(1,0)}(z) \, \overline{ F^{(0,1)}(z)}\big]} \,\big\vert\, F(z)=0 \big]
    = \mathbb{E}\big[
    \bigabs{\Im[ - Z_1  \overline{ Z_2}]}
    \big]
  \end{equation*}
  where $Z = (Z_1\, Z_2)\in \mathbb{C}^2$ is a circularly symmetric Gaussian random variable with covariance matrix given by $\Omega$ in \eqref{eq_Omega1}.
Let us assume initially that $\Omega$ is non-singular and denote
$\newdelta := \det(\Omega)$, so that $\newdelta > 0$.
Hence,
\begin{align*}
\pi\rho_1 =\frac{1}{\pi^2 \newdelta}  \int_{\mathbb{C}^2} |\Im(z \overline{w})| e^{- (z,w)^* \Omega^{-1} (z,w)}
dA(z) dA(w).
\end{align*} 
We use the following formula \cite[eq.~(4.32)]{berry2000phase}
\begin{align*}
\abs{x}
=\frac{1}{\pi}\int_{-\infty}^{+\infty} \left(1-\cos(xt)\right) \frac{dt}{t^2}
=
\frac{1}{\pi} \Re\bigg(
\int_{-\infty}^{+\infty} \left(1-e^{itx} \right) \frac{dt}{t^2}\bigg),
\end{align*}
where the last integral is to be understood as a principal value.
By Lemma \ref{lemma_integral},
\begin{align*}
\pi\rho_1 
& =
\frac{1}{\pi \newdelta} 
\Re\bigg(
\int_{-\infty}^{+\infty} 
\frac{1}{\pi^2} \int_{\mathbb{C}^2} 
\big(1-e^{it\Im(z\bar w)} \big) 
e^{- (z,w)^* \Omega^{-1} (z,w)}
dA(z) dA(w)
\frac{dt}{t^2}
\bigg)
\\
& =
\frac{1}{\pi \newdelta} 
\Re\bigg(
\int_{-\infty}^{+\infty} 
\bigg(
\newdelta-
\det\bigg(
\Omega^{-1}+\frac{t}{2} J
\bigg)^{-1}
\bigg)
\frac{dt}{t^2}
\bigg)
\\
& =
\frac{1}{\pi} 
\Re\bigg(
\int_{-\infty}^{+\infty} 
\bigg(
1-
\det\bigg(
I +
\frac{t}{2}\Omega J
\bigg)^{-1}
\bigg)
\frac{dt}{t^2}
\bigg).
\end{align*}
Using Lemma \ref{lemma_HH}, we note that the off-diagonal element
$\gamma$ in $\Omega$ satisfies 
\begin{align}
\label{eq_gamma_gamma_bar}
\bar\gamma-\gamma=2i,
\end{align}
and, hence, 
\begin{align*}
\det\bigg(I + \frac{t}{2}\Omega J\bigg)
=\begin{vmatrix}
-\frac{t}{2} \gamma + 1 & \frac{t}{2} \alpha \\
-\frac{t}{2} \beta & \frac{t}{2} \bar\gamma +1
\end{vmatrix}
=1+\frac{t}{2}(\bar\gamma-\gamma)
-\frac{t^2}{4}\big(\abs{\gamma}^2-\alpha\beta\big)
=1+it
+\frac{t^2}{4}
\newdelta.
\end{align*}
Therefore
\begin{align*}
&\frac{1}{\pi} 
\Re\bigg(
\int_{-\infty}^{+\infty} 
\bigg(
1-
\det\bigg(
I +
\frac{t}{2}\Omega J
\bigg)^{-1}
\bigg)
\frac{dt}{t^2}
\bigg)
\\
& \qquad =
\frac{1}{\pi} 
\Re\bigg(
\int_{-\infty}^{+\infty} 
\bigg(
1-
\frac{1}{1+it
+\frac{t^2}{4}
\newdelta}
\bigg)
\frac{dt}{t^2}
\bigg)
\\
& \qquad = 
\frac{1}{\pi} 
\Re\bigg(
\int_{-\infty}^{+\infty} 
\bigg(
\frac{1
+\frac{t^2}{2}
\newdelta 
+\frac{t^4}{16}\newdelta^2
+ t^2 
- 1+it
-\frac{t^2}{4}
\newdelta
}{
1
+\frac{t^2}{2}
\newdelta 
+\frac{t^4}{16}\newdelta^2
+ t^2}
\bigg)
\frac{dt}{t^2}
\bigg)
\\
& \qquad = 
\frac{1}{\pi} 
\int_{-\infty}^{+\infty} 
\frac{
\frac{1}{4}\newdelta 
+\frac{t^2}{16}\newdelta^2
+ 1
}{
1
+\frac{t^2}{2}
\newdelta 
+\frac{t^4}{16}\newdelta^2
+ t^2}
dt
\\
& \qquad = 
\frac{1}{\pi} 
\int_{-\infty}^{+\infty} 
\frac{
t^2\newdelta^2
+ 4\newdelta 
+ 16
}{
t^4 \newdelta^2
+ 8 t^2\newdelta 
+ 16 t^2
+ 16
}
dt\,.
\end{align*}
Let us write
$t^4\newdelta^2
+
8 t^2 \big(\newdelta+2
\big)+16=\newdelta^2(t^2+\lambda)(t^2+\mu)$
with
\begin{align}
\label{eq_aaa}
\newdelta^2 \lambda \mu = 16, \quad
\newdelta^2(\lambda+\mu)=8(\newdelta+2).
\end{align}
This shows that $\lambda,\mu \in (0,+\infty)$; we may assume that
$\lambda \geq \mu$. A direct calculation further shows that
\begin{align}
\label{eq_bbb}
\sqrt{\lambda}-\sqrt{\mu}=\sqrt{\lambda\mu}, \quad
\sqrt{\lambda}+\sqrt{\mu}=\sqrt{\newdelta+1}\sqrt{\lambda\mu},
\quad
\lambda-\mu=\lambda \mu\sqrt{\newdelta+1}
=\tfrac{16}{\newdelta^2}\sqrt{\newdelta+1}.
\end{align}
Using \eqref{eq_aaa} we write
\begin{align*}
{t^2}
\newdelta^2
+4\newdelta
+16 &=
t^2 \newdelta^2
+4(\newdelta+2)
+8=
\tfrac{\newdelta^2}{2}
\left((t^2+\lambda) + (t^2+\mu)\right)
+\tfrac{8}{\lambda-\mu}\left((t^2+\lambda)-(t^2+\mu)\right)
\\
&
=
\newdelta^2 \left(
\big(\tfrac{1}{2}+\tfrac{1}{2\sqrt{\newdelta+1}}\big)
(t^2+\lambda) + 
\big(\tfrac{1}{2}-\tfrac{1}{2\sqrt{\newdelta+1}}\big)
(t^2+\mu)
\right).
\end{align*}
Hence, with the aid of \eqref{eq_bbb}, we compute
\begin{align*}
\pi\rho_1&=
\frac{1}{\pi}
\int_{-\infty}^\infty
\left(\frac{1}{2}-\frac{1}{2\sqrt{\newdelta+1}}\right)\frac{1}{t^2+\lambda} + 
\left(\frac{1}{2}+\frac{1}{2\sqrt{\newdelta+1}}\right)\frac{1}{t^2+\mu} \, dt
\\
&=
\left(\frac{1}{2}-\frac{1}{2\sqrt{\newdelta+1}}\right)\frac{1}{\sqrt{\lambda}} + 
\left(\frac{1}{2}+\frac{1}{2\sqrt{\newdelta+1}}\right)\frac{1}{\sqrt{\mu}}
\\
&=
\frac{1}{\sqrt{\lambda\mu}}
\left(\frac{1}{2}\left(\sqrt{\lambda}+\sqrt{\mu}\right)
+\frac{1}{2\sqrt{\newdelta+1}}
\left(\sqrt{\lambda}-\sqrt{\mu}\right)
\right)
\\
&=
\left(\frac{1}{2}\sqrt{\newdelta+1}
+\frac{1}{2\sqrt{\newdelta+1}}
\right)
=
\frac{\newdelta+2}{2\sqrt{\newdelta+1}},
\end{align*}
as claimed.

Finally, if $\Omega$ is singular, we let $Z^\tau := Z + \tau X$ with $X$ an independent standard circularly symmetric random vector on $\mathbb{C}^2$ and $\tau \in (0,1)$. 
Then $Z^\tau$ has covariance $\Omega^\tau = \Omega + \tau I$ and the calculation above shows that
\begin{align*}
\lim_{\tau \to 0+}
\mathbb{E}\big[
    \bigabs{\Im[ - Z^\tau_1  \overline{ Z^\tau_2}]}
    \big]
    = \lim_{\tau \to 0+} \frac{\det(\Omega + \tau I)+2}{2\sqrt{\det(\Omega + \tau I)+1}}
    = 1\,.
\end{align*}
Furthermore,  by continuity, $\bigabs{\Im[ - Z^\tau_1  \overline{ Z^\tau_2}]} \longrightarrow \bigabs{\Im[ - Z_1  \overline{ Z_2}]}$ almost surely
and
\begin{align}\label{eq_before}
\bigabs{\Im[ - Z^\tau_1  \overline{ Z^\tau_2}]} \leq \abs{Z^\tau_1} \cdot \abs{Z^\tau_2} \leq \big(\abs{Z_1}+ \abs{X_1}\big)
\cdot \big(\abs{Z_2}+ \abs{X_2}\big), \qquad \tau \in (0,1).
\end{align}
Since $Z_1, Z_2, X_1, X_2$ are normal, they are square integrable with respect to the underlying probability, and thus the right hand side of \eqref{eq_before}
is integrable. Hence, by dominated convergence,
\begin{align*}
\frac{1}{{\pi}}
=
\frac{1}{{\pi}} 
\lim_{\tau \to 0+}
\mathbb{E}\big[
    \bigabs{\Im[ - Z^\tau_1  \overline{ Z^\tau_2}]}
    \big]
=
\frac{1}{{\pi}} 
\mathbb{E}\big[
    \bigabs{\Im[ - Z_1  \overline{ Z_2}]}
    \big]
= \rho_1\, .
\end{align*}
This completes the proof.
\end{proof}

As an application of Theorem \ref{th_unsigned_H}, we   derive a simplified expression for radial twisted kernels.

\begin{proof}[Proof of Corollary \ref{coro_one_point_radial}]
We use Theorem \ref{th_unsigned_H}. We first note that $H(0)=P(0)=1$ and compute
\begin{align*}
H^{(1,0)}(z) & = 2 x P'\big(\abs{z}^2\big)
\\
H^{(1,1)}(z) & = 4 x y P''\big(\abs{z}^2\big)
\\
H^{(2,0)}(z) & = 4 x^2  P''\big(\abs{z}^2\big) + 2 P'\big(\abs{z}^2\big).
\end{align*}
Hence, $H^{(1,0)}(0)=H^{(1,1)}(0)=0$, and
$H^{(2,0)}(0)=2P'(0)$. 
By symmetry,
$H^{(0,1)}(0)=0$, and $H^{(0,2)}(0)=2P'(0)$.
This gives that $\Omega$, as defined in \eqref{eq_Omega1}, is
\begin{align}
\Omega
=
\begin{bmatrix}
-2P'(0) & -i \\
i & -2P'(0)
\end{bmatrix}
\end{align} 
and its determinant $\newdelta$
is
\begin{align*}
\newdelta
= \big(-2P'(0)\big)^2 - 1 
= \big( 2 P'(0)+1 \big) \big(2 P'(0)-1\big).
\end{align*}
Because $\newdelta\geq 0$, this implies that 
either 
$P'(0)\geq 1/2$ or $P'(0)\leq -1/2$.
However, 
also the minor $-2P'(0)$ of $\Omega$ has to be nonnegative, i.e., $P'(0)\leq 0$ which implies that $P'(0)\leq -1/2$ is the only valid option.
To obtain $\rho_1$ in \eqref{eq_rho_H}, we calculate
\begin{align*}
\newdelta + 2 &= 4 P'(0)^2 +1,\\
\sqrt{\newdelta+1} &= -2P'(0),
\end{align*}
and \eqref{eq_rho_H} simplifies to \eqref{eq_one_point_radial}.
\end{proof}

Finally, we derive the first intensity of charged zeros.

\begin{proof}[Proof of  Theorem \ref{th_signed_H}]
We proceed along the lines of the proof of Theorem \ref{th_unsigned_H}.
This time we use Kac-Rice's formula \eqref{eq_KRg}, which 
gives that $\rho^{\chsign}_1$ is the following constant:
\begin{align*}
\rho^{\chsign}_1 
= \frac{1}{{\pi}} \mathbb{E} \big[
{-} \Im\big[  F^{(1,0)}(z) \, \overline{ F^{(0,1)}(z)}\big] \,\big\vert\, F(z)=0 \big].
\end{align*}
By Proposition~\ref{prop_reg}\ref{en:jacfcondstat},
  \begin{equation*}
    \mathbb{E} \big[
{-} \Im\big[  F^{(1,0)}(z) \, \overline{ F^{(0,1)}(z)}\big] \,\big\vert\, F(z)=0 \big]
    = \mathbb{E}\big[
    \Im[ - Z_1  \overline{ Z_2}]
    \big]
  \end{equation*}
  where $Z = (Z_1\, Z_2)\in \mathbb{C}^2$ is a circularly symmetric Gaussian random variable with covariance matrix given by $\Omega$ in \eqref{eq_Omega1}.
  Thus, 
\begin{align*}
\rho^{\chsign}_1 
= {-}\frac{1}{{\pi}}  \Im (\gamma) = \frac{1}{{\pi}} 
\end{align*}
where $\gamma$ is the covariance $\mathbb{E}\big[ Z_1  \overline{ Z_2}\big]$  in \eqref{eq_Omega1}.
\end{proof}

\section{Charge Fluctuations}\label{sec_var}
\subsection{Sufficient conditions for hyperuniformity}
The following lemma gives sufficient conditions for the hyperuniformity of the charged zero set \eqref{eq_charged_measure}.
These are formulated in terms of the first intensity of the uncharged zero set $\rho_1$ --- which is constant by Theorem \ref{th_signed_H} --- and the semi-charged
two-point intensity defined in Lemma \ref{lemma_kr}, cf.~\eqref{eq_KR2}.

\begin{lemma}\label{lemma_hyper}
Let $F$ be a GWHF with twisted kernel $H$  satisfying the standing assumptions.
Suppose that
\begin{align}\label{eq_hyper1}
&\int_{\mathbb{C}} (1+\abs{z}) \biggabs{\frac{1}{\pi^2}-\tau^{\chsign}_2(z)} dA(z) < \infty,
\end{align}
and
\begin{align}
\label{eq_hyper2}
&\int_{\mathbb{C}} \left(\frac{1}{\pi^2}-\tau^{\chsign}_2(z)\right) dA(z)
= \rho_1.
\end{align}
Then there exists $C>0$ such that for all $z_0 \in \mathbb{C}$,
\begin{align}\label{eq_aaa1}
\var\big[\mathcal{Z}^{\chsign}_F(B_R(z_0))\big]
\leq C R, \qquad R>0,
\end{align}
and
\begin{align}\label{eq_aaa2}
\tfrac{1}{R} \var\big[\mathcal{Z}^{\chsign}_F(B_R(z_0))\big]
\to
\int_{\mathbb{C}} \abs{z} \left(\frac{1}{\pi^2}-\tau^{\chsign}_2(z)\right) dA(z),
\qquad \mbox{ as } {R \to \infty}.
\end{align}
\end{lemma}
\begin{proof}
We let $z_0 \in \mathbb{C}$ and use \eqref{eq:tauint}
and Theorem \ref{th_signed_H} to compute
\begin{align*}
&\var\Big[\mathcal{Z}^{\chsign}_F (B_R(z_0))\Big]
=
\mathbb{E}\left[
\big(\mathcal{Z}^{\chsign}_F (B_R(z_0))\big)^2
\right]-
\left(
\mathbb{E}\left[
\mathcal{Z}^{\chsign}_F (B_R(z_0))
\right] \right)^2
\\
&\qquad=
\mathbb{E}\left[
\big(\mathcal{Z}^{\chsign}_F (B_R(z_0))\big)^2
-\mathcal{Z}_F (B_R(z_0))
\right]-
\left(
\mathbb{E}\left[
\mathcal{Z}^{\chsign}_F (B_R(z_0))
\right] \right)^2
+
\mathbb{E}\left[
\mathcal{Z}_F (B_R(z_0))\right]
\\
&\qquad=
\int_{B_R(z_0)\times B_R(z_0)} \tau_2^{\chsign}(z-w) dA(z) dA(w)
-
\left(
\int_{B_R(z_0)} \rho_1^{\chsign} \,dA(z)
\right)^2
+
\int_{B_R(z_0)} \rho_1 \, dA(z)
\\
&\qquad= \int_{B_R(z_0)\times B_R(z_0)} \left(\tau_2^{\chsign}(z-w) 
-
\frac{1}{\pi^2}\right)
dA(z) dA(w)
+
\rho_1 \abs{B_R(z_0)}.
\end{align*}
In terms of the function $\funct(z) := \frac{1}{\pi^2 \rho_1} - \frac{1}{\rho_1} \tau^{\chsign}_2(z)$,
the expression for the variance reads
\begin{align*}
\frac{1}{\rho_1} \var\Big[\mathcal{Z}^{\chsign}_F (B_R(z_0))\Big]
&= \abs{B_R(z_0)} - \int_{B_R(z_0)\times B_R(z_0)} \funct(z-w)
dA(z) dA(w).
\end{align*}
The last expression measures the average deviation within the disk $B_R(z_0)$ between the indicator function of that disk and its convolution with $\varphi$. Precise estimates are given in Lemma \ref{lemma_bv} --- whose proof is deferred to Section \ref{sec_aux}. The hypotheses of Lemma \ref{lemma_bv} are met due to \eqref{eq_hyper1} and \eqref{eq_hyper2}, and we readily obtain \eqref{eq_aaa1} and \eqref{eq_aaa2}.
\end{proof}

\subsection{Computations for radial twisted kernels}
For radial twisted kernels, the following proposition provides an expression for the integrals in \eqref{eq_hyper2}. We use the notation of Lemma \ref{lemma_hyper}.

\begin{prop}\label{prop_variance_radial}
Let $F$ be a GWHF with twisted kernel $H$ satisfying the standing assumptions. Assume further that $H(z)=P\big(\abs{z}^2\big)$,
where $P\colon \mathbb{R} \to \mathbb{R}$ is $C^2$, $P(0)=1$, and
\begin{align}\label{eq_vvvv}
\sup_{r \geq 0}
\big(
\abs{P(r^2)} + \abs{P'(r^2)} + \abs{P''(r^2)}
\big) r^4  < \infty.
\end{align}
Then
\begin{align}
\label{eq_integral_variance_4}
\int_{\mathbb{C}}  
\biggabs{\frac{1}{\pi^2}-\tau^{\chsign}_2(z)} dA(z) 
& < \infty,
\\
\label{eq_integral_variance_3}
\int_{\mathbb{C}} \abs{z}
\biggabs{\frac{1}{\pi^2}-\tau^{\chsign}_2(z)} dA(z) 
& < \infty,
\\
\label{eq_integral_variance}
\int_{\mathbb{C}} \bigg(\frac{1}{\pi^2} - \tau^{\chsign}_2(z) \bigg) dA(z)
& = 
- \frac{1}{\pi}
\bigg(
 P'(0)
+ \frac{1  }{ 4 P'(0)}
\bigg),
\\
\label{eq_integral_variance_2}
\int_{\mathbb{C}} \abs{z} \bigg(\frac{1}{\pi^2}-\tau^{\chsign}_2(z)\bigg) dA(z) 
& = \frac{1}{\pi}  
\int_0^\infty 
\frac{2 r^2 P'(r^2)^2}{1-P(r^2)^2}
dr.
\end{align}
\end{prop}

\begin{proof}
\noindent {\bf Step 1. }\emph{(Calculation of the semi-charged 2-point intensity)}.	

We recall that the covariance structure of $(F, F^{(1,0)}, F^{(0,1)})$ is given 
in Section~\ref{sec:covstruct}.
Here, we further need the covariance structure of $\big(F(z), F^{(1,0)}(z), F^{(0,1)}(z), F(w), F^{(1,0)}(w), F^{(0,1)}(w)\big)$.
We first compute
\begin{align*}
  H^{(1,0)}(z) & = 2 x P'(\abs{z}^2),
  \\
  H^{(0,1)}(z) & = 2 y P'(\abs{z}^2),
  \\
  H^{(2,0)}(z) & = 2 P'(\abs{z}^2) + 4 x^2 P''(\abs{z}^2),
  \\
  H^{(0,2)}(z) & = 2 P'(\abs{z}^2) + 4 y^2 P''(\abs{z}^2),
  \\
  H^{(1,1)}(z) & = 4 x y P''(\abs{z}^2).
\end{align*}
Thus, the covariance matrix of $\big(F(z), F^{(1,0)}(z), F^{(0,1)}(z), F(w), F^{(1,0)}(w), F^{(0,1)}(w)\big)$ can be calculated as
\begin{align*}
  \begin{pmatrix}
    \Gamma(z) & \Gamma(z,w) \\
    \Gamma(z,w)^* & \Gamma(w)
  \end{pmatrix}
\end{align*}
where $\Gamma(z)$ is given by \eqref{eq:covfz} and
simplifies to 
\begin{align*}
  \Gamma(z) 
  & =
  \begin{pmatrix}
    1
    &  
     iy 
    & 
     - ix 
    \\
    - iy 
    & 
    y^2 -2 P'(0)
    & 
    -i- xy
    \\
     ix 
    &
    i- xy
    &
    x^2-2 P'(0)
  \end{pmatrix}
\end{align*}
and $\Gamma(z,w)$ is
\begin{align*}
  \Gamma(z,w) 
  & =
  e^{i (yu-xv)}
  \left(
  P\big(\abs{z-w}^2\big)
  \begin{pmatrix}
    1
    &  
     iy 
    & 
     - ix 
    \\
    - iv 
    & 
    yv 
    & 
    -i- xv
    \\
     iu 
    &
    i- uy
    &
    xu
  \end{pmatrix}
  \right.
  \\*
  & \quad
  +
  2 P'(\abs{z-w}^2)
  \begin{pmatrix}
    0
    &  
    -(x-u) 
    & 
    -(y-v)  
    \\
    (x-u)  
    & 
    -1+ i (x-u) (y +v) 
    & 
    i (y-v)v - i(x-u) x
    \\
    (y-v)
    &
    i (y-v)y- i(x-u) u
    &
    -1 - i (y-v)(x +u)
  \end{pmatrix}
  \\*
  & \quad
  \left. + 
  4 P''(\abs{z-w}^2)
  \begin{pmatrix}
    0
    &  
    0
    & 
    0  
    \\
    0
    & 
    (x-u)^2 
    & 
    -(x-u) (y-v) 
    \\
    0
    &
    -(x-u) (y-v) 
    &
    (y-v)^2 
  \end{pmatrix}
  \right).
\end{align*}
We are interested in quantities of the form
\begin{align*}
\mathbb{E}\big[ h\big(F^{(1,0)}(z), F^{(0,1)}(z), F^{(1,0)}(w), F^{(0,1)}(w) \big) \, \big\vert \, (F(z), F(w)) = (0,0) \big].
\end{align*}
Following Remark~\ref{lemma_condition}, this conditional expectation is $\mathbb{E}\big[ h(Z) \big]$ where $Z\in \mathbb{C}^4$ is a circularly symmetric Gaussian random variable with covariance matrix:
\begin{equation}
  \Omega(z,w) 
  = 
  A-B C^{-1} B^*\,.
\end{equation}
Here,
\begin{equation}
  A 
  =
  \begin{pmatrix}
    \Gamma_{2,3;2,3}(z)
    & \Gamma_{2,3;2,3}(z,w)
    \\
    \Gamma_{2,3;2,3}(z,w)^*
    & \Gamma_{2,3;2,3}(w)
  \end{pmatrix},
\end{equation}
\begin{align*}
  B 
  & =
  \begin{pmatrix}
    \Gamma_{2,3;1}(z)
    & \Gamma_{2,3;1}(z,w)
    \\
    \Gamma_{1;2,3}(z,w)^*
    & \Gamma_{2,3;1}(w)
  \end{pmatrix}
  \\
  & =
  \begin{pmatrix}
    -iy 
    & e^{i (yu-xv)} \big(- iv P + 2(x-u)P'\big)
    \\
     ix
    & e^{i (yu-xv)} \big( iu P + 2(y-v)P'\big)
    \\
    e^{- i (yu-xv)} \big(- iy P - 2(x-u)P'\big)
    &-iv 
    \\
     e^{- i (yu-xv)} \big(ix P - 2(y-v)P'\big)
    &  iu
  \end{pmatrix}
\end{align*}
where
$\Gamma_{i,j; k,l}$ is the submatrix of $\Omega$ containing the rows $i$ and $j$ and columns
$k$ and $l$,
and
\begin{equation}
  C 
  =
  \begin{pmatrix}
    1 & e^{i (yu-xv)} P \\
     e^{- i (yu-xv)}  P & 1
  \end{pmatrix}.
\end{equation}
Thus,
\begin{equation}
  C^{-1}
  =
  \frac{1}{1- P ^2}
  \begin{pmatrix}
    1 & -e^{i (yu-xv)} P \\
     -e^{- i (yu-xv)}  P & 1
  \end{pmatrix}
\end{equation}
with the convention that $P$, $P'$, and $P''$ are understood to be evaluated at $\abs{z-w}^2$.

By Lemma \ref{lemma_kr},
\begin{align}\label{eq_pi2r} 
\begin{aligned}
\pi^2 \tau_2^{\chsign}(z-w) 
&=
\frac{1}{1-  P^2 }  E,
\end{aligned}
\end{align}
where
\begin{align}\label{def_E}
E :=
\mathbb{E}\big[
\jac F(z) \, \jac F(w) \,\big|\, F(z)=F(w)=0 \big].
\end{align}
We now invoke a variant of Wick's formula (Isserlis' theorem), proved in Lemma \ref{lemma_wick} below, and obtain
\begin{align*}
E
= -\frac{1}{2} \Re \big[ 
\Omega_{1,2}\Omega_{3,4}
+
\Omega_{1,4}\Omega_{3,2}
- 
\Omega_{2,1}\Omega_{3,4}
- 
\Omega_{2,4}\Omega_{3,1}
\big],
\end{align*}
With this information we can calculate explicitly $E$
to obtain
\begin{align*}
\pi^2 \tau_2^{\chsign}(z-w) = 1 + I'(\abs{z-w}^2), 
\end{align*}
with
\begin{align}\label{eq_I}
I(s) &= 
\frac{s \big(2P'(s)^2+\frac{3}{2} P(s)^2\big)}{1-P(s)^2}
+ \frac{2 s^2 P(s)P'(s)}{(1-P(s)^2)^2}
\,.
\end{align}
Section \ref{sec_calc_1} contains detailed calculations, which also give the estimate
\begin{align}
\label{eq_bound_I}
\sup_{r \geq 0} (1+r^4) \bigabs{I'(r^2)} < \infty.
\end{align}

\noindent {\bf Step 2. }\emph{(Conclusions)}.
We first verify \eqref{eq_integral_variance_4} and \eqref{eq_integral_variance_3}; this then implies that the integrals in
\eqref{eq_integral_variance} and \eqref{eq_integral_variance_2} are absolutely convergent.
To this end, we use \eqref{eq_bound_I} and estimate
\begin{align*}
\int_{\mathbb{C}} 
\biggabs{\frac{1}{\pi^2} - \tau^{\chsign}_2(z) } dA(z)
& =
\frac{1}{\pi^2}
\int_{\mathbb{C}} \abs{I' (|z|^2)} dA(z)
\\
& = \frac{2}{\pi} \int_0^\infty r \abs{I'\big(r^2\big)}\,  dr
\\
&
\lesssim \int_0^\infty \frac{r}{1+r^4} \, dr < \infty,
\end{align*}
and, similarly,
\begin{align*}
\int_{\mathbb{C}} 
\abs{z}
\biggabs{\frac{1}{\pi^2} - \tau^{\chsign}_2(z) } dA(z)
&
\lesssim \int_0^\infty \frac{r^2}{1+r^4} \, dr < \infty.
\end{align*}

For \eqref{eq_integral_variance} first note that
\begin{align*}
\lim_{s\to 0}\frac{s }{1-P(s)^2}
=
-\frac{1}{2P'(0)},
\end{align*}
while $\lim_{s \to \infty} I(s)=0$ by \eqref{eq_vvvv}. Hence,
\begin{align*}
\int_{\mathbb{C}} \bigg(\frac{1}{\pi^2} - \tau^{\chsign}_2(z) \bigg) dA(z) 
&=
-\frac{1}{\pi^2} \int_{\mathbb{C}} I'\big(\abs{z}^2\big) dA(z)
\\
&
=
-\frac{1}{\pi} \int_0^\infty I'(s) ds
\\
&= \frac{1}{\pi} \lim_{s \to 0} I(s)
\\
&= 
\frac{1}{\pi}
\bigg(
-\frac{ 2P'(0)^2+\frac{3}{2}}{2P'(0)}
+ \frac{1  }{ 2 P'(0)}
\bigg)
\\
&= 
- \frac{1}{\pi}
\bigg(
 P'(0)
+ \frac{1  }{ 4 P'(0)}
\bigg).
\end{align*}
For \eqref{eq_integral_variance_2}, integration by parts gives
\begin{align*}
\int_{\mathbb{C}} 
\abs{z}
\Big(\frac{1}{\pi^2} - \tau^{\chsign}_2(z) \Big) dA(z)
&=
-\frac{1}{\pi^2} \int_{\mathbb{C}} \abs{z} I'\big(\abs{z}^2\big) dA(z)
\\
&
= -\frac{2}{\pi} \int_0^\infty r^2 I'\big(r^2\big) dr
\\
&
= -\frac{1}{\pi} \int_0^\infty r I'\big(r^2\big) 2r dr
\\
&
=\frac{1}{\pi}  \int_0^\infty I\big(r^2\big) dr.
\end{align*}
A direct calculation shows that
\begin{align}\label{eq_Ir2}
I(r^2) = 
\frac{2 r^2 P'(r^2)^2}{1-P(r^2)^2}
+ \frac{d}{dr}\bigg[\frac{ r^3   P(r^2)^2   
}{ 2(1-P(r^2)^2 )} 
\bigg].
\end{align}
Hence, by \eqref{eq_vvvv},
\begin{align*}
\int_{\mathbb{C}} 
\abs{z}
\Big(\frac{1}{\pi^2} - \tau^{\chsign}_2(z) \Big) dA(z) 
&= 
\frac{1}{\pi}
\int_0^\infty
\frac{2 r^2 P'(r^2)^2}{1-P(r^2)^2}  
dr
+ 
\frac{1}{\pi} \bigg[
\frac{ r^3   P(r^2)^2   
}{ 2(1-P(r^2)^2 )}
\bigg]_{r=0}^{\infty}
\\
& = \frac{1}{\pi}  
\int_0^\infty 
\frac{2 r^2 P'(r^2)^2}{1-P(r^2)^2}
dr,
\end{align*}
as claimed in \eqref{eq_integral_variance_2}.
\end{proof}

\subsection{Proof of Theorem \ref{th_signed_variance}}
We now derive the main result on hyperuniformity of charge.
We invoke Lemma \ref{lemma_hyper}. Condition \eqref{eq_hyper1} is satisfied as shown in Proposition
\ref{prop_variance_radial}, \eqref{eq_integral_variance_3}, while \eqref{eq_hyper2} is seen to hold by comparing the explicit expressions given in Corollary \ref{coro_one_point_radial} and Proposition \ref{prop_variance_radial}. The asymptotic value of the variance in \eqref{eq_aaa2} is computed in \eqref{eq_integral_variance_2}.\qed

\section{Examples and applications}\label{sec_app}

\subsection{The short time Fourier transform of white noise}\label{sec_tf}
Let $g\colon \mathbb{R} \to \mathbb{C}$ be a Schwartz function. 
As a first step towards the definition of the short-time Fourier transform of white noise, we consider its distributional formulation. 
For a Schwartz function $f\colon  \mathbb{R} \to \mathbb{C}$, we write \eqref{eq_stft} as
\begin{align}\label{eq_stft_2}
V_g f(x,y) = \langle f, \tfs(x,y) g \rangle,
\end{align}
where $\tfs(x,y) g$ denotes the \emph{time-frequency shift}\footnote{Time-frequency shifts are usually denoted $\pi(x,y)$; we prefer $\tfs(x,y)$ to avoid confusion with the numerical constant.}
\begin{align*}
\tfs(x,y) g(t) = e^{2 \pi i y t} g(t-x), \qquad t \in \mathbb{R}.
\end{align*}
We define the STFT of a distribution $f \in \mathcal{S}'(\mathbb{R})$ by \eqref{eq_stft_2}, using the distributional interpretation of the $L^2$ inner product $\langle \cdot, \cdot \rangle$, and note that this defines a smooth function on $\mathbb{R}^2$. 
The \emph{adjoint short-time Fourier transform}
$V^*_g\colon  \mathcal{S}(\mathbb{R}^2) \to \mathcal{S}(\mathbb{R})$,
\begin{align*}
V^*_g \varphi (t) = \int_{\mathbb{R}^2} \varphi(x,y) \,\tfs(x,y) g(t) \, dx dy, 
\end{align*}
provides the following concrete description of the distributional STFT:
\begin{align}\label{eq_stft_3}
\langle V_g f, \varphi \rangle = \langle f, V^*_g \varphi \rangle,
\qquad f \in \mathcal{S}'(\mathbb{R}), \quad
\varphi \in \mathcal{S}(\mathbb{R}^2).
\end{align}
See \cite[Chapter 11]{charlybook} for more background on the STFT of distributions.

Let $\noise$ be complex white noise on $\mathbb{R}$, that is, $\noise = \frac{1}{\sqrt{2}} \frac{d}{dt} \big(W_1 + i W_2\big)$, where $W_1$ and $W_2$ are independent copies of the Wiener process (Brownian motion with almost surely continuous paths), and the derivative is taken in the distributional sense. 
The short-time Fourier transform of complex white noise is the random function:
\begin{align*}
V_g \, \noise(z) = \langle \noise, \tfs(x,y) g \rangle, \qquad z=x+iy \in \mathbb{C};
\end{align*}
see \cite{MR4047541, bh} for other definitions and a comprehensive discussion on their equivalence. Then $V_g \, \noise$ is Gaussian because, as a consequence of \eqref{eq_stft_3}, for any Schwartz function $\varphi \in \mathcal{S}(\mathbb{R}^2)$,
$\langle V_g \, \noise, \varphi \rangle = \langle \noise, V^*_g \varphi \rangle$
is normally distributed. In addition, $V_g \, \noise$ is circularly symmetric, as, for any $\theta \in \mathbb{R}$, $e^{i \theta} \cdot V_g \, \noise =  V_g \, \big( e^{i \theta} \cdot \noise \big)  \sim V_g \, \noise$. One readily verifies that
\begin{align}\label{eq_cov_stft}
\mathbb{E} \big[ 
V_g \,\noise(z) \cdot \overline{V_g \,\noise(w)}
\big] = \langle
\tfs(u,v) g , \tfs(x,y) g \rangle,\qquad z,w \in \mathbb{C}.
\end{align}

The following lemma relates the STFT of white noise and GWHFs.
\begin{lemma}\label{lemma_stft_gwhf}
	Let $g\colon  \mathbb{R} \to \mathbb{C}$ be a Schwartz function normalized by $\norm{g}_2=1$, 
	and consider the short-time Fourier transform of complex white noise, twisted and scaled as follows:
	\begin{align*}
	F(z) := e^{-i xy}  \cdot V_g \, \noise \big(\bar{z}/\sqrt{\pi}\big), \qquad z=x+iy.
	\end{align*}
	Then $F$ is a GWHF with twisted kernel
	\begin{align*}
	H(z) = e^{-i xy} \cdot
	V_g g \big(\bar{z}/\sqrt{\pi} \big), \qquad z=x+iy,
	\end{align*}
	and the standing assumptions are satisfied.
	
	In addition, the zero set of $V_g \, \noise$ has a first intensity $\rho_{1,g}$ related to that of the zero set of $F$ by	
	\begin{align}
	\label{eq_F}
	\rho_{1,g} = \pi \rho_1.
	\end{align}	
\end{lemma}
\begin{proof}
	$F$ is Gaussian and circularly symmetric because $V_g \,\noise$ is.
	Using \eqref{eq_cov_stft}, we inspect the covariance of $F$:
	\begin{align*}
	\mathbb{E}
	\big[
	F(z)  \overline{F(w)}
	\big]
  & = 
  e^{i (uv-xy)}
  \Big\langle
  \tfs\big(\tfrac{u}{\sqrt{\pi}},-\tfrac{v}{\sqrt{\pi}}\big) g , \tfs\big(\tfrac{x}{\sqrt{\pi}},-\tfrac{y}{\sqrt{\pi}}\big) g 
  \Big\rangle
  \\
	& = 
	e^{i (uv-xy)}
	\int_{\mathbb{R}}
	g\big(t-\tfrac{u}{\sqrt{\pi}}\big) \overline{g\big(t-\tfrac{x}{\sqrt{\pi}}\big)}e^{-2\sqrt{\pi} i (v-y) t}  dt
	\\
	&=
	e^{i (uv-xy)}
	\int_{\mathbb{R}} 
	g(t) \overline{g\big(t+\tfrac{u-x}{\sqrt{\pi}}\big)}
	e^{- 2 \sqrt{\pi} i (v-y)(t+u/\sqrt{\pi})} dt
	\\
	&=
	e^{i (yu-xv)}
	e^{-i (x-u)(y-v)}
	\int_{\mathbb{R}} g(t)  \overline{g\big(t-\tfrac{(x-u)}{\sqrt{\pi}}\big)}
	e^{- 2 \pi i \sqrt{\pi}{(v-y)}{}t} dt
	\\
	&=
	e^{i\Im(z \bar w)}
	H(z-w).
	\end{align*}
	We now verify the standing assumptions.
	Since $g$ is Schwartz, $H$ is $C^\infty$, and
	\eqref{eq_HC2} and \eqref{eq_cont_paths} hold. The
	normalization condition \eqref{eq_HH} is indeed satisfied since $H(0)= V_gg(0)=\norm{g}^2_2=1$. To check the non-degeneracy condition \eqref{eq_non_deg} note first that,
	by Cauchy-Schwarz,
	\begin{align*}
	\abs{H(z)} = \abs{\ip{g}{\tfs(x/\sqrt{\pi},-y/\sqrt{\pi}) g}} \leq \norm{g}_2^2=1=H(0).
	\end{align*}
	If equality holds for some $z=x+iy$, then there exists $\lambda \in \mathbb{C}$ such that
	\begin{align*}
	g = \lambda \tfs(x/\sqrt{\pi},y/\sqrt{\pi}) g.
	\end{align*}
	This implies,
	\begin{align*}
	|g(t)| = | \lambda g(t-x/\sqrt{\pi})|, \qquad t \in \mathbb{R}.
	\end{align*}
	Since $g \in L^2(\mathbb{R}) \setminus \{0\}$, we must have $x=0$. Hence, 
	\begin{align*}
	g(t) = \lambda e^{-2 \sqrt{\pi} i y t} g(t), \qquad t \in \mathbb{R},
	\end{align*}
	which implies $y=0$, since $g \not \equiv 0$. Hence $z=0$.
	
	Finally, since $F(z)=0$ if and only if $V_g \,\mathcal{N}(\overline{z}/\sqrt{\pi})=0$, \eqref{eq_F} follows.
\end{proof}

\subsection{Calculation of the first intensity}
We now apply our results to the short-time Fourier transform of complex white noise.
\begin{proof}[Proof of Theorem \ref{th_real_stft}]
	We consider the functions $F$ and $H$ as in Lemma \ref{lemma_stft_gwhf}, and the first intensities of their zero sets, $\rho_1$ and $\rho_{1,g}$, related by \eqref{eq_F}.
	We calculate
	\begin{align*}
	H(0) 
	& =V_g g(0) = \norm{g}^2_2=1,
	\\
	H^{(1,0)}(0)
	& =\frac{1}{\sqrt{\pi}} (V_g g)^{(1,0)}(0)
	=-\frac{1}{\sqrt{\pi}}
	\int_{\mathbb{R}} g(t) \overline{g'(t)} dt
	= -\frac{1}{\sqrt{\pi}} i c_4,
	\\
	H^{(0,1)}(0)
	& =\frac{-1}{\sqrt{\pi}}(V_g g)^{(0,1)}(0)
	= 2\sqrt{\pi}i
	\int_{\mathbb{R}} t \abs{g(t)}^2 dt
	= 2\sqrt{\pi}i c_1,
	\\
	H^{(2,0)}(0)
	& = \frac{1}{\pi} (V_gg)^{(2,0)}(0)
	=\frac{1}{\pi}
	\int_{\mathbb{R}} g(t) \overline{g''(t)} dt
	=-\frac{1}{\pi}\int_{\mathbb{R}} \abs{g'(t)}^2 dt
	= -\frac{1}{\pi} c_3,
	\\
	H^{(0,2)}(0)
	& = \frac{1}{\pi} (V_gg)^{(0,2)}(0)
	= - 4 \pi
	\int_{\mathbb{R}} t^2 \abs{g(t)}^2 dt
	= - 4 \pi c_2,
	\\
	H^{(1,1)}(0)
	& = -i V_gg(0) - \tfrac{1}{\pi} (V_gg)^{(1,1)}(0)
	=-	i - 2 i\int_{\mathbb{R}} tg(t) \overline{g'(t)} dt
	=
	2 \Im\bigg(\int_{\mathbb{R}} tg(t) \overline{g'(t)} dt\bigg)
	= 2 c_5,
	\end{align*}
	where we used that, by Lemma \ref{lemma_HH}, $H^{(1,1)}(0) \in \mathbb{R}$. Note also that $c_4 \in \mathbb{R}$ by Lemma \ref{lemma_HH},
	while, clearly, $c_1, c_2, c_3,  c_5 \in \mathbb{R}$.	
	Thus, \eqref{eq_newdelta} is given by
	\begin{align*}
	\newdelta 
	& = \det
	\begin{bmatrix}
	\frac{1}{\pi} c_3
	-\frac{1}{\pi} c_4^2
	& -2 c_5 -i -   2 c_1 c_4 \\
	-2 c_5 + i -   2 c_1 c_4
	& 
	4 \pi c_2
	-4 \pi  c_1^2
	\end{bmatrix}
	\\
	& =
	( c_3
	- c_4^2)
	(4  c_2
	-4 c_1^2)
	-(4 c_5^2 + 4 c_1^2 c_4^2 + 8 c_1 c_4 c_5 +1)
	\\
	& =
	4  c_2 c_3
	-4  c_2 c_4^2
	-4 c_1^2 c_3
	- 4 c_5^2 
	- 8 c_1 c_4 c_5 
	-1
	\\
	& =
	4 ( c_2 - c_1^2)c_3
	-4  c_2 c_4^2
	- 4 c_5^2 
	- 8 c_1 c_4 c_5 
	-1
	\end{align*}
	and
	Theorem \ref{th_unsigned_H}, together with \eqref{eq_F}, yield
	\begin{align}
	\rho_{1,g} 
	= \pi \rho_1 
	= 
	\frac{
		4 ( c_2 - c_1^2)c_3
		-4  c_2 c_4^2
		- 4 c_5^2 
		- 8 c_1 c_4 c_5 
		+1}{4  
		\sqrt{
			( c_2 - c_1^2)c_3
			-c_2 c_4^2
			-c_5^2 
			-2c_1 c_4 c_5
	}},
	\label{eq:rhoonestftgen}
	\end{align}
	as claimed.
	
Finally, if $g$ is real valued, integration by parts gives
\begin{align*}
\int_{\mathbb{R}} g(t) g'(t) dt = - \int_{\mathbb{R}} g'(t) g(t) dt,
\end{align*}
showing that $c_4=0$, while clearly $c_5=0$.
\end{proof}

\subsection{The uncertainty principle for zeros}
In order to show that generalized Gaussian windows minimize the expected numbers of zeros of the STFT with complex white noise, we first show that the corresponding intensities are invariant under certain transformations that preserve the class of Gaussians.

\begin{lemma}\label{lemma_stft_change}
	Let $g\colon \mathbb{R} \to \mathbb{C}$ be a Schwartz function, and $x_0, \xi_0, \xi_1 \in \mathbb{R}$. Let
	\begin{align*}
	g_1(t) := e^{2 \pi i \left(\xi_0 t + \xi_1 t^2\right)}\cdot g(t-x_0), \qquad t \in \mathbb{R}.
	\end{align*}
	Then the first intensities of the zero sets of $V_{g} \, \mathcal{N}$ and  $V_{g_1} \, \mathcal{N}$ coincide:
	\begin{align*}
	\rho_{1,g}=\rho_{1,g_1}.
	\end{align*}
\end{lemma}
\begin{proof}
	We proceed in two steps, and exploit different properties of the STFT. We first assume that $\xi_1=0$ and use the so-called \emph{covariance of the STFT under time-frequency shifts}:
	\begin{align*}
	V_{g_1} f (x,y) = V_{\tfs(x_0, \xi_0) g} \,f (x,y) =
	e^{2 \pi i \xi_0 x} \cdot V_g f(x+x_0, y+\xi_0),
	\end{align*}
	which can be verified by direct calculation or deduced from \cite[Lemma 3.1.3]{charlybook}. Applying this formula to each realization of complex white noise
	$f=\noise$, we deduce that
	$\mathcal{Z}_{F_{g_1}}$ and $\mathcal{Z}_{F_{g}}$ are related by a deterministic translation:
	$\mathcal{Z}_{F_{g_1}} = \mathcal{Z}_{F_{g}}
	(\cdot - x _0, \cdot - \xi_0)$. Hence,
	$\rho_{1,g}=\rho_{1,g_1}$.
	
	We now assume that $\xi_0=x_0=0$, so that $g_1$ and $g$ are related by the unitary operator
	$U\colon L^2(\mathbb{R}) \to L^2(\mathbb{R})$,
	\begin{align*}
	g_1 (t) = U g(t) = e^{2 \pi i \xi_1 t^2} g(t).
	\end{align*}
	The operator $U$ is also an isomorphism on the spaces of Schwartz functions and tempered distributions. For a distribution $f$, we use the formula
	\begin{align}\label{eq_U}
	V_{g_1} \big(U f\big) (x,y) = e^{-2 \pi i \xi_1 x^2}\cdot 
	V_g f (S(x,y)), \qquad S(x,y)=(x, y - 2 x \xi_1),
	\end{align}
	which can be readily verified or deduced as special case of the \emph{symplectic covariance of the STFT} \cite[Chapter 4]{folland89} \cite[Section 9.4]{charlybook}. Let $\noise$ be complex white noise; then so is $U \noise$ (both generalized Gaussian processes have the same stochastics). 
	In addition, by Lemma \ref{lemma_stft_gwhf} and Theorem \ref{th_unsigned_H}, the zero sets of $V_g \, \noise$ and
	$V_{g_1} \, \big(U \noise\big)$ have first intensities and these are constant. Hence, for any Borel set $E \subseteq \mathbb{R}^2$,
	by \eqref{eq_U},
	\begin{align*}
	\rho_{1,g_{1}} |E| &=
	\mathbb{E} \big[ 
	\# \{ (x,y) \in E: V_{g_1} \noise(x,y) = 0 \}
	\big]
  \\
	& =
	\mathbb{E} \big[ 
	\# \{ (x,y) \in E: V_{g_1} \big(U \noise \big) (z) = 0 \}
	\big]
	\\&=
	\mathbb{E} \big[ 
	\# \{ (x,y) \in E: V_{g} \big(U \noise \big) (S(x,y)) = 0 \}
	\big]
	\\&=
	\mathbb{E} \big[ 
	\# \{ (x',y') \in S(E): V_{g} \big(U \noise \big) (x',y') = 0 \}
	\big]
	\\&= \rho_{1,g_{1}} |S(E)| = \rho_{1,g_{1}} |E|,
	\end{align*}
	as $S$ is a linear map with determinant equal to $1$.
	
	Finally, the general case without assumptions on $\xi_0$, $\xi_0$ and $x_0$ follows from the discussed special cases by successively considering the effect of the time-frequency shift $\tfs(x_0,\xi_0)$ and quadratic modulation $U$. 
\end{proof}

We can now prove the announced uncertainty principle for zero sets.
\begin{proof}[Proof of Theorem \ref{th_stft_up}]
	We use the notation of the proof of Theorem \ref{th_real_stft}. Recall the relation \eqref{eq_F}. As shown in Theorem \ref{th_unsigned_H} and its proof, $\rho_1$ as given by \eqref{eq_rho_H} satisfies $\rho_1 \geq 1/\pi$ and achieves the value $1/\pi$ exactly when
	$\newdelta=0$. We now describe the functions attaining that minimum.
	
	\noindent {\bf Step 1}. \emph{(Special minimizers).}
	We consider first windows $g$ such that $c_1=c_4=c_5=0$. For such windows the minimality condition $\newdelta=0$ reads $4 \cdot c_2 \cdot c_3 = 1$ and means that $g$ saturates \emph{Heisenberg's uncertainty relation}:
	\begin{align}\label{eq_hup}
	\int_\RR t^2 \abs{g(t)}^2 dt
	\cdot
	\int_\RR |g'(t)|^2 dt
	= \frac{1}{4}
	= \frac{1}{4} \norm{g}_2^2.
	\end{align}
	By Heisenberg's uncertainty principle, the solutions to \eqref{eq_hup} are exactly the Gaussians:
	\begin{align}\label{eq_stg}
	\frac{\lambda}{\sqrt{\sigma}} e^{-\tfrac{\pi}{\sigma^2} t^2}, 
	\qquad t \in \mathbb{R},
	\end{align}
	with $\sigma>0$ and $|\lambda|=2^{1/4}$;
	see, e.g., \cite[Corollary 1.35]{folland89}. Thus, we conclude that the Gaussians \eqref{eq_stg} achieve the minimal intensity $\rho_{1,g}=1$, and that these are the only minimizers among (unit norm) windows with $c_1=c_4=c_5=0$.
	
	\noindent {\bf Step 2}. \emph{(General minimizers).}
	Suppose that $\rho_{1,g}$ is minimal and consider
	\begin{align}\label{eq_g1}
	g_1(t) := e^{-i \big(\xi_0 t + \tfrac{\xi_1}{2} t^2\big)}\cdot g(t-x_0), \qquad t \in \mathbb{R},
	\end{align}
	with $x_0, \xi_0,\xi_1 \in \mathbb{R}$. Let $d_1, \ldots, d_5$ be the uncertainty constants defined similarly to $c_1, \ldots, c_5$ but with respect to $g_1$. We now show that it is possible to choose the parameters $x_0, \xi_0,\xi_1$ so that $d_1=d_4=d_5=0$. First, choosing $x_0 := -c_1$, we get
	\begin{align*}
	d_1 = \int_{\mathbb{R}} t |g(t-x_0)|^2 dt = 
	c_1 + x_0 = 0.
	\end{align*}
	Similarly,
	\begin{align*}
	i d_4 = \int_{\mathbb{R}} 
	g(t-x_0) \cdot \left[(i\xi_0 + i \xi_1 t) \overline{g(t-x_0)} + \overline{g'(t-x_0)}\right] dt= 
	i \xi_0 + i \xi_1 d_1 + i c_4 = i \xi_0 + i c_4,
	\end{align*}
	so it suffices to take $\xi_0 := -c_4$, which is indeed a real number as proved in Theorem \ref{th_real_stft}. Finally,
	\begin{align*}
	d_5 &= \Im \left(
	\int_{\mathbb{R}} 
	t g(t-x_0) \cdot \left[(i\xi_0 + i \xi_1 t) \overline{g(t-x_0)} + \overline{g'(t-x_0)}\right] dt \right)
	\\
	&= \Im \big[i \xi_0 d_1 + i \xi_1 d_2\big] + x_0 c_4 + c_5
	=  \xi_1 d_2 + x_0 c_4 + c_5.
	\end{align*}
	As $g \not\equiv 0$, $d_2>0$. In addition, $c_4, c_5 \in \mathbb{R}$. Hence, $\xi_1$ can be chosen so that $d_5=0$.
	
	By Lemma \ref{lemma_stft_change}, $\rho_{1, g_1}=\rho_{1, g}$ is also minimal. Thus, by Step 1, $g_1$ must be a Gaussian \eqref{eq_stg}, and, therefore, $g$ is a generalized Gaussian \eqref{eq_gaussian}.
	
	Conversely, if $g$ is a generalized Gaussian \eqref{eq_gaussian}, then we can choose $\xi_0, \xi_1, x_0 \in \mathbb{R}$ so that $g_1$ takes the form \eqref{eq_stg}. Hence, by Step 1 and Lemma \ref{lemma_stft_change},
	$\rho_{1,g}=\rho_{1,g_1}=1$.
\end{proof}

\subsection{Hermite windows}\label{sec_stft_hermite}
We now consider Hermite functions
\begin{align}  \label{eq_hermite}
h_{r}(t) = \frac{2^{1/4}}{\sqrt{r!}}\left(\frac{-1}{2\sqrt{\pi}}\right)^r
e^{\pi t^2} \frac{d^r}{dt^r}\left(e^{-2\pi t^2}\right), \qquad r \geq 0,
\end{align}
as windows for the STFT. According to Lemma \ref{lemma_stft_gwhf}, $F(x+iy) := e^{-i xy}  V_{h_r} \, \noise (\overline{z}/\sqrt{\pi})$ is a GWHF with twisted covariance kernel $H(z) = e^{-i xy}
V_{h_r} {h_r} (\overline{z}/\sqrt{\pi})$. The kernel can be calculated explicitly in terms of Laguerre polynomials
\begin{align}\label{eq_lag0}
L_n(t)=\sum_{j=0}^n (-1)^j \binom{n}{j} \frac{t^j}{j!},
\end{align}
by the following formula
\begin{align}\label{eq_twisted_kernel_herm}
H(z) = L_{r}(|z|^2) e^{-\tfrac{1}{2}|z|^2},
\end{align}
known as the \emph{Laguerre connection} \cite[Theorem (1.104)]{folland89}. We thus obtain a simple expression for the first intensity of the zeros of the STFT of complex noise with Hermite windows.

\begin{proof}[Proof of Corollary \ref{coro_intro_her}]
We write $H(z)=P(|z|^2)$ with $P(t)=L_r(t) e^{-t/2}$. By Lemma \ref{lemma_stft_gwhf}, $H$ satisfies the standing assumptions. We can therefore apply Corollary \ref{coro_one_point_radial}.
Inspecting \eqref{eq_lag0} we
obtain
\begin{align*}
P'(0) 
& =  L'_{r}(0) -\frac{1}{2} L_{r}(0)
=-r-\frac{1}{2}.
\end{align*}
Using \eqref{eq_F}, we conclude
\begin{align*}
\rho_{1,h_r} = \pi \rho_{1} =
- \left(
P'(0) + \frac{1}{4 P'(0)}
\right) = r + \frac{1}{2} + \frac{1}{4r+2}.
\end{align*}
\end{proof}

\subsection{Derivatives of Gaussian entire functions}\label{sec_pure}
Let $G_0$ be a Gaussian entire function, that is, a  circularly symmetric random function with correlation kernel,
\begin{align}\label{eq_a}
\mathbb{E}
\left[
G_0(z) \cdot \overline{G_0(w)}
\right]=e^{z \bar{w}},
\end{align}
and consider the iterated covariant derivatives
\begin{align}\label{eq_true_app}
G(z)=\big(\bar{\partial}^*\big)^{q-1} G_0 =
\big(\bar{z} - \partial \big)^{q-1} G_0,
\end{align}
where $q \in \mathbb{N}$.
$G$ is called a \emph{Gaussian poly-entire function of pure type}. The following lemma provides an identification with a GWHF.

\begin{lemma}\label{lemma_true_gwhf}
	Let $G$ be a Gaussian poly-entire function of pure-type, as in \eqref{eq_true_app}. Then
	\begin{align*}
F(z) = \frac{e^{-\tfrac{1}{2} |z|^2}}{\sqrt{(q-1)!}}  \cdot G(z), \qquad z \in \mathbb{C},
	\end{align*}
	is a GWHF with twisted kernel
	\begin{align}\label{eq_H_pure}
	H(z)= L_{q-1}(|z|^2) \cdot e^{-\tfrac{1}{2}|z|^2}
	\end{align}
	satisfying the standing assumptions. Here, $L_n$ denotes the Laguerre polynomial \eqref{eq_lag0}.
\end{lemma}
\begin{proof}
We consider the \emph{complex Hermite polynomials} $H_{k,j}(z, \bar{z})$ defined by:
\begin{align*}
H_{k,q-1}(z,\bar{z}) := (\bar{z} - \partial_z)^{q-1} \big[ z^k \big], \qquad k \geq 0.
\end{align*}
Conjugating the last equation we obtain:
\begin{align*}
\overline{H_{k,q-1}(z,\bar{z})} =
(z - \partial_{\bar z})^{q-1} \big[ \bar{z}^k \big].
\end{align*}

We combine \eqref{eq_a} and \eqref{eq_true_app},
expand $e^{z \bar{w}}$ into series, and compute
\begin{align*}
\mathbb{E}
\left[
G(z) \cdot \overline{G(w)}
\right]&=
\big(\bar{z} - \partial_z \big)^{q-1}
\big(w - \partial_{\bar{w}} \big)^{q-1}\big[e^{z\bar{w}}\big]
\\
&= \sum_{k \geq 0}
\frac{1}{k!} H_{k,q-1}(z,\bar{z}) \overline{H_{k,q-1}(w,\bar{w})}
=(q-1)! \, L_{q-1}(|z-w|^2) e^{z\bar{w}},
\end{align*}
where the last equality is proved in \cite[Equation 3.19]{gh13}; see also \cite[Proposition 3.7]{gh13} and
\cite[Section 2]{is16}.

Hence,
\begin{align*}
\mathbb{E}
\left[
F(z) \cdot \overline{F(w)}
\right]&=\exp\left[-\tfrac{1}{2}|z|^2-\tfrac{1}{2}|w|^2 + z \bar{w} \right] L_{q-1}\big(|z-w|^2\big)
\\
&=e^{i \Im(z \bar{w})} \cdot H(z-w),
\end{align*}
as desired. Finally, note that $F$ is also the GWHF associated in Section \ref{sec_stft_hermite} with the STFT with Hermite window $h_{q-1}$. Hence, the standard assumptions hold by Lemma \ref{lemma_stft_gwhf}.
\end{proof}

\subsection{Gaussian poly-entire functions}\label{sec_gpef}
We now look into Gaussian poly-entire function of full type (cf. Example \ref{ex_poly}). These are defined as
\begin{align}\label{eq_gfull}
G = \sum_{k=0}^{q-1} \frac{1}{\sqrt{k!}} \big(\bar{\partial}^*\big)^k G_k
\end{align}
where $G_0, \ldots, G_{q-1}$ are independent Gaussian entire functions, and $q$ is called the order of $G$. The following lemma identifies $G$ with a GWHF, by means of the generalized Laguerre polynomial
\begin{align*}
L^{(1)}_n(t) = \sum_{k=0}^n L_k(t).
\end{align*}

\begin{lemma}\label{lemma_poly_gwhf}
Let $G$ be a Gaussian poly-entire function of full type of order $q$, as in \eqref{eq_gfull}. Then
$F(z) = q^{-1/2} \cdot e^{-\tfrac{1}{2} |z|^2} \cdot G(z)$ is a GWHF with twisted kernel
\begin{align}\label{eq_H_poly}
H(z)= q^{-1} L^{(1)}_{q-1}(|z|^2) e^{-\tfrac{1}{2}|z|^2}
\end{align}
satisfying the standing assumptions.
\end{lemma}
\begin{proof}
By Lemma \ref{lemma_true_gwhf},
$F = \frac{1}{\sqrt{q}}\sum_{k=0}^{q-1} F_k$, where $F_1, \ldots, F_{q-1}$ are independent GWHF with respective twisted kernels $H_k(z) = L_k(|z|^2) e^{-\tfrac{1}{2}|z|^2}$. Due to independence,
\begin{align*}
\mathbb{E} \big[ F(z) \cdot \overline{F(w)} \big]
= \frac{1}{q} \sum_{k=0}^{q-1} e^{i\Im(z \bar{w})} H_k(|z-w|^2) = e^{i\Im(z \bar{w})} H(|z-w|^2).
\end{align*}
By Lemma \ref{lemma_true_gwhf}, each twisted kernel $H_k(z) = L_k(|z|^2) e^{-\tfrac{1}{2}|z|^2}$ satisfies \eqref{eq_pd} and \eqref{eq_HH}, and therefore so does its average $H$. In addition, \eqref{eq_HC2} and \eqref{eq_cont_paths} are satisfied as $H \in C^\infty(\mathbb{R}^2)$.
\end{proof}
As an application, we obtain the following.
\begin{proof}[Proof of Theorem \ref{th_poly}]
By Lemmas \ref{lemma_true_gwhf} and \ref{lemma_poly_gwhf}, we can apply Corollary \ref{coro_one_point_radial} with $P(t) = L_{q-1}(t) e^{-t/2}$ or
$P(t)=q^{-1} L^{(1)}_{q-1}(t) e^{-t/2}$. In the first case (pure type), the calculation was carried out in the proof of Corollary \ref{coro_intro_her} (where $r=q-1$). For the second case (full type), 
we note that
$L^{(1)}_{q-1}(0)=q$, while
\begin{align*}
\frac{d}{dt}L^{(1)}_{q-1}(0)=\sum_{k=0}^{q-1} 
L'_{k}(0)=\sum_{k=0}^{q-1} (-k)
= - \frac{q(q-1)}{2}.
\end{align*}
We thus compute,
	\begin{align*}
	P'(0) &= \frac{1}{q} \left( \frac{d}{dt}L^{(1)}_{q-1}(0) -\frac{1}{2} L^{(1)}_{q-1}(0) \right)
	\\
	&=\frac{1}{q} \left( -\frac{q(q-1)}{2} -\frac{q}{2} \right)=-\frac{q}{2},
	\end{align*}
	and, therefore,
	\begin{align*}
	\rho_1 
	= - \frac{1}{\pi}  \left(
	P'(0) + \frac{1}{4 P'(0)}
	\right)=\frac{1}{2\pi}  \left(
	q + \frac{1}{q}
	\right).
	\end{align*}
\end{proof}

\subsection{Charges}\label{sec_app_charges}
We start with the following general observation.
\begin{lemma}\label{lemma_factor}
Let $F,G:\mathbb{C} \to \mathbb{C}$ be $C^1$ in the real sense, and $z_0 \in \mathbb{C}$. If $F(z_0)=0$ and $G(z_0) \not=0$, then the charges of $F$ and $F\cdot G$ at $z_0$ coincide.
\end{lemma}
\begin{proof}
Using \eqref{eq_Jim} we see that the charge of $F \cdot G$ at $z_0$ is
\begin{align*}
\sgn \Big[ \jac(F\cdot G)(z_0) \Big] &=
- \sgn \Big[ \Im\Big[  (F\cdot G)^{(1,0)}(z) \cdot \overline{ (F\cdot G)^{(0,1)}(z)}\Big] \Big]
\\
&=- \sgn \Big[ \Im\Big[  G(z_0) \cdot F^{(1,0)}(z) \cdot \overline{G(z_0)} \cdot \overline{ F^{(0,1)}(z)}\Big] \Big]
\\
&=- \sgn \Big[ \Im\Big[  |G(z_0)|^2 \cdot F^{(1,0)}(z) \cdot \overline{ F^{(0,1)}(z)}\Big] \Big]
\\
&=- \sgn \Big[ \Im\Big[  F^{(1,0)}(z) \cdot \overline{ F^{(0,1)}(z)}\Big] \Big],
\end{align*}
which is also the charge of $F$ at $z_0$.
\end{proof}

We first apply Theorem \ref{th_signed_H} to the short-time Fourier transform, and obtain formulas in terms of \eqref{eq_newcharge}.

\begin{proof}[Proof of Corollary \ref{coro_charge_stft}]
By Lemma \ref{lemma_stft_gwhf}, the short-time Fourier transform of complex white noise
can be identified with a GWHF by the transformation
\begin{align*}
F(z) := e^{-i xy}  V_g \, \noise (\bar{z}/\sqrt{\pi}), \qquad z=x+iy.
\end{align*}
At a zero $\zeta=a+ib$,
\begin{align*}
F^{(1,0)}(\zeta)&= \frac{e^{-i ab}}{\sqrt{\pi}} \big(V_g \, \mathcal{N}\big)^{(1,0)}(a/\sqrt{\pi},-b/\sqrt{\pi}),
\\
F^{(0,1)}(\zeta)&= -\frac{e^{-i ab}}{\sqrt{\pi}} \big(V_g \, \mathcal{N}\big)^{(0,1)}(a/\sqrt{\pi},-b/\sqrt{\pi}),
\end{align*}
and, consequently, $\jac F(\zeta)=\frac{1}{\pi} \Im \Big[\,\big(V_g \, \mathcal{N}\big)^{(1,0)}(a/\sqrt{\pi},-b/\sqrt{\pi}) \cdot \overline{\big(V_g \, \mathcal{N}\big)^{(0,1)}(a/\sqrt{\pi},-b/\sqrt{\pi})} \,\Big]$. 

Applying Theorem \ref{th_signed_H} 
with the change of variable $z= \bar{\zeta}/\sqrt{\pi}$
we obtain
\begin{align*}
\mathbb{E} \Big[ 
\sum_{z \in E, \, V_g\, \mathcal{N} (z) = 0} 
\newcharge_z
\Big]
&= \mathbb{E} \Big[ 
\sum_{\zeta \in \sqrt{\pi} \bar{E}, \, F(\zeta) = 0} 
\newcharge_{\bar{\zeta}/\sqrt{\pi}} 
\Big]
\\
&=\mathbb{E} \Big[ 
\sum_{\zeta \in \sqrt{\pi} \bar{E}, \, F(\zeta) = 0} 
\sgn \jac F(\zeta)\Big]
\\
& =
\frac{1}{\pi} \big| \sqrt{\pi} \bar{E}\big| 
\\
& =
 |E|,
\end{align*}
as claimed.
\end{proof}
For the STFT of white noise with a Hermite window \eqref{eq_hermite}, the twisted kernel is given in \eqref{eq_twisted_kernel_herm}, and we can apply Theorem \ref{th_signed_variance}
with
\begin{align*}
P(t) = L_{r}(t) e^{-t/2}. 
\end{align*}
After a change of variables as in the proof of Corollary \ref{coro_charge_stft}, we obtain
\begin{align*}
\mathrm{Var}  \bigg[ 
\sum_{z \in B_R(z_0), \, V_g\, \mathcal{N} (z) = 0} 
\newcharge_z
\bigg] \leq C_r R,
\end{align*}
while 
\begin{align*}
\frac{1}{R} \mathrm{Var}
\bigg[ 
\sum_{z \in B_R(z_0), \, V_g\, \mathcal{N} (z) = 0} 
\newcharge_z
\bigg]
\to  
\pi^{-1/2}
\int_0^\infty 
\frac{2 t^2 P'(t^2)^2}{1-P(t^2)^2}
dt,
\qquad \mbox{ as } {R \to \infty},
\end{align*}
uniformly on $z_0$.

Finally, we note that we can also apply Theorems~\ref{th_signed_H} and~\ref{th_signed_variance} to poly-entire functions. 
Let $G$ be a Gaussian poly-entire function of pure-type,
as in \eqref{eq_true_app}. According to Lemma \ref{lemma_true_gwhf}, the function
\begin{align*}
F(z) = \frac{e^{-\frac{1}{2} |z|^2}}{\sqrt{(q-1)!}} \cdot G(z),
\end{align*}
is a GWHF. By Lemma \ref{lemma_factor}, the charges of $F$ and $G$ at a zero $\zeta$ coincide:
\begin{align*}
\charge_{\zeta} = \sgn(\jac F(\zeta)) = \sgn(\jac G(\zeta)).
\end{align*}
A similar argument applies to poly-entire functions of full-type (cf. Example \ref{ex_poly} and Section \ref{sec_gpef}). Hence, Theorem \ref{th_signed_H} shows that the first intensity of the charged zeros of $G$ is $1/\pi$. Similarly, Theorem \ref{th_signed_variance} applies to $G$ and concrete expressions for the asymptotic charged particle variance can be obtained with the polynomials
\begin{align*}
P(r) = 
\begin{cases}
e^{-r/2} \cdot L_{q-1}(r) &\mbox{pure-type \eqref{eq_true_app}}\\
\frac{1}{q} \cdot e^{-r/2} \cdot L^{(1)}_{q-1}(r) &\mbox{full-type \eqref{eq_gfull}}
\end{cases}.
\end{align*}

\subsection{First derivatives of GEF}\label{sec_fd}
We now interpret the statistics of zeros of Gaussian pure poly-entire functions of order 1, and show how they recover the well-known first order statistics of critical points of weighted magnitudes of Gaussian entire functions (cf. Examples \ref{ex_true}).

Let $G$ be a Gaussian entire function as in Example \ref{ex_gaf} and consider its amplitude $A(z) = e^{-\frac{1}{2}|z|^2} |G(z)|$. Then, by \eqref{eq_wa}, the critical points of $A$ are exactly the zeros of the GWHF $F(z)=e^{-\frac{1}{2}|z|^2} \bar{\partial}^* G(z)$. By Theorem \ref{th_poly} (with $q=2$), the first intensity of the critical points of $A$ is therefore $5/3 \cdot 1/\pi$.

Second, consider a critical point $z_0$ of $A$. Then, by Proposition
\ref{prop_reg}, with probability one, $z_0$ is not a zero of $G$, and near $z_0$ we can write
$G(z)=L(z)^2$ with $L$ analytic. Hence,
\begin{align*}
2 \partial A = A^{(1,0)} - i A^{(0,1)} = -{\frac{\overline{L}}{L}} \cdot F.
\end{align*}
As the factor ${{\overline{L}}/{L}}$ is smooth (in the real sense) and non-zero near $z_0$, we conclude 
by Lemma \ref{lemma_factor} that the charge of $F$ at $z_0$ is
\begin{align*}
\charge_z = \sgn \Big[ \big[A^{(1,1)}\big]^2 - A^{(2,0)} A^{(0,2)} \Big],
\end{align*}
that is, 
the opposite of the sign of the determinant of the Hessian matrix of $A$ at $z_0$. Hence, $\charge_z=1$ if $z_0$ is a saddle point of $A$, while $\charge_z=-1$ if $A$ has a local maximum at $z_0$ (while local minima are excluded, as they are zeros of $G$ \cite[Section 8.2.2]{gafbook}). Thus, by Theorem \ref{th_signed_H}, the first intensity of the quantity ``saddle points $-$ local maxima'' is $1/\pi$. Combining this with the first intensity of the total critical points, we conclude that the first intensity of the local maxima of $A$ is $1/3 \cdot 1/\pi$ whereas that of the saddle points is $4/3 \cdot 1/\pi$.

While the calculation of first intensities of different kinds of critical points of $G$ is well-known --- they follow for example as the limit of more precise results for polynomial spaces in \cite[Corollary 5]{MR2104882} --- the hyperuniformity of the statistics of ``saddle points $-$ local maxima'' is, to the best of our knowledge, a novel consequence of Theorem \ref{th_signed_variance}.

\section{Conclusions and outlook}\label{sec_conc}

We introduced the notion of twisted stationarity for an ensemble of random functions and obtained basic statistics for their zeros. In comparison to the model case of translation invariant Gaussian entire functions, a novel element is found: GWHF may either preserve or reverse orientation around a zero, and zero statistics are thus augmented with the new attribute of charge.

While our result on hyperuniformity of charge is a first step in the exploration of repulsion between zeros of GWHF, as it shows that a universal form of screening is observed at large scales, many important questions remain open. First, Theorem \ref{th_signed_variance} was obtained under the assumption that the twisted kernel is radial, which means that statistics are rotationally invariant. We do not know if hyperuniformity of charge holds also for non-radial twisted kernels. Second, no variance estimates were derived for uncharged zeros. We conjecture that the uncharged number variance grows like the perimeter of the observation disk. Finally, numerical experience suggests that the repulsion between zeros of the same charge is stronger than that between oppositely charged ones, but we do not yet have formal statistics justifying that claim.

The short-time Fourier transform of white noise is a case in point application of our results, because they open the door to the use of non-Gaussian windows. This new freedom has prospective applications in signal processing which we expect to develop in future work. Indeed, when analyzing a signal,  one can often choose the STFT window, and the potentially rich zero statistics that we derived hold \emph{simultaneously} for all such choices.

\section{Auxiliary results}\label{sec_aux}
\subsection{Computations with Gaussians}
\begin{lemma}\label{lemma_integral}
Let $\Omega \in \mathbb{C}^{2\times 2}$ be positive definite and $t \in \mathbb{R}$. Then
\begin{align}\label{eq_integral}
\frac{1}{\pi^2} \int_{\mathbb{C}^2} 
e^{- (\bar z,\bar w) \,\Omega\, (z,w)^t} e^{it \Im(z\bar w)}
dA(z) \,dA(w)
=
\frac{1}{\det\big(\Omega + \tfrac{t}{2} J\big)}.
\end{align}
\end{lemma}
\begin{proof}
Write
\begin{equation*}
\Omega
=
\begin{pmatrix}
	a & b+id
  \\
  b-id & c
\end{pmatrix},
\end{equation*}
fix $a>0$, $b \in \mathbb{R}$,
and consider both sides of \eqref{eq_integral} as functions of the complex variable $\xi = \frac{t}{2}+id$.
For $\xi\in i \mathbb{R}$ (i.e., $t=0$) and $d^2<ac-b^2$ (i.e., $\Omega$ positive definite),  \eqref{eq_integral} holds because it expresses the fact that the probability density of a complex Gaussian is normalized.
We will show that both sides of \eqref{eq_integral} are analytic functions on the domain
\begin{align*}
\mathcal{A} = 
\big\{ \xi \in \mathbb{C}:
\big(\Im [\xi ]\big)^2 < ac-b^2 \big\}.
\end{align*}
To this end, we first rewrite
\begin{equation*}
  \frac{1}{\pi^2} \int_{\mathbb{C}^2} 
e^{- (\bar z,\bar w) \,\Omega\, (z,w)^t} e^{it \Im(z\bar w)}
dA(z) \,dA(w)
=
\frac{1}{\pi^2} \int_{\mathbb{C}^2} 
e^{- a z \bar z - c w \bar w + (\xi-b) z \bar w + (-\xi-b) w \bar z} 
d(z) \,dA(w).
\end{equation*}
Here, the integrand is an analytic function in $\xi$.
To show the analyticity of the integral, we note that for any compact subset $\mathcal{C} \subseteq \mathcal{A}$, we have $\vartheta_{\mathcal{C}}:=\sup_{\xi \in \mathcal{C}} \big(\Im [\xi ]\big)^2 < ac-b^2$.
Thus, the absolute integrand satisfies
\begin{equation*}
  \frac{1}{\pi^2} \int_{\mathbb{C}^2} \bigabs{
e^{- (\bar z,\bar w) \,\Omega\, (z,w)^t} e^{it \Im(z\bar w)}}
dA(z) \,dA(w)
=
\frac{1}{\pi^2} \int_{\mathbb{C}^2} 
e^{- (\bar z,\bar w) \,\Omega\, (z,w)^t}
dA(z) \,dA(w)
\leq 
\frac{1}{ac-b^2- \vartheta_{\mathcal{C}}}.
\end{equation*}
Hence, the absolute integral is uniformly bounded for $\xi \in \mathcal{C}$.
Applying Morera's theorem and Fubini's theorem, we can conclude that the integral is analytic as well.

The right-hand side of   \eqref{eq_integral} can be rewritten as
\begin{equation}
  \frac{1}{\det\big(\Omega + \tfrac{t}{2} J\big)}
  = 
  \frac{1}{ac-(b+\xi)(b-\xi)}
\end{equation}
and is also analytic in $\xi$ as long as $ac\neq (b+\xi)(b-\xi)$.
In particular, for $\xi \in \mathcal{A}$ we have that 
$\Re[ac-(b+\xi)(b-\xi)]= ac-b^2+\frac{t^2}{4}-d^2 \geq ac-b^2-d^2>0$.
Hence, both sides of \eqref{eq_integral} are analytic on $\mathcal{A}$ and coincide on the set $\mathcal{A} \cap i \mathbb{R}$. 
By the identity theorem of analytic functions, they thus coincide on $\mathcal{A}$.
\end{proof}

\begin{lemma}\label{lemma_wick}
Let $v$ be a $4$-dimensional circularly symmetric complex Gaussian vector with covariance matrix $\Omega$.  
Then
\begin{align*}
\E \Big[ \Im(v_1 \bar v_2) \cdot \Im(v_3 \bar v_4 ) \Big] 
& = 
-\frac{1}{2} \Re \Big[ 
\Omega_{1,2}\Omega_{3,4}
+
\Omega_{1,4}\Omega_{3,2}
- 
\Omega_{2,1}\Omega_{3,4}
- 
\Omega_{2,4}\Omega_{3,1}
\Big].
\end{align*}
\end{lemma}
\begin{proof}
By Wick's formula (see, e.g., \cite[Lemma 2.1.7]{gafbook}),
we have
\begin{align*}
\E \Big[ \Im(v_1 \bar v_2) \cdot \Im(v_3 \bar v_4 ) \Big] 
& = 
-\frac{1}{2} \Re \,\E \Big[v_1 v_3 \bar v_2 \bar v_4 
- v_2 v_3 \bar v_1 \bar v_4 \Big] 
\\
& = 
-\frac{1}{2} \Re \Big[ \per( \Omega_{1,3;2,4})  - 
\per (\Omega_{2,3; 1,4}) \Big],
\\
& = 
-\frac{1}{2} \Re \Big[ 
\Omega_{1,2}\Omega_{3,4}
+
\Omega_{1,4}\Omega_{3,2}
- 
\Omega_{2,1}\Omega_{3,4}
- 
\Omega_{2,4}\Omega_{3,1}
\Big],
\end{align*}
where $\per$ is the permanent and $\Omega_{i,j; k,l}$ is the submatrix of $\Omega$ containing the rows $i$ and $j$ and columns
$k$ and $l$. 
\end{proof}

\subsection{Regularization by convolution}

\begin{lemma}
\label{lemma_bv}
Let $\funct \colon \mathbb{C} \to \mathbb{R}$ be an integrable function that satisfies
\begin{align*}
\int_\mathbb{C} \funct(z) \, dA(z) = 1, \qquad
C_\funct := \int_{\mathbb{C}} \abs{z} \abs{\funct(z)} \, dA(z) < \infty.
\end{align*}
Then there exists a universal constant $C>0$ such that
for all $z_0 \in \mathbb{C}$,
\begin{align}
\label{eq_bv_1}
&\biggabs{\abs{B_r(z_0)} - 
\int_{B_r(z_0)} (\funct * 1_{B_r(z_0)})(z) \, dA(z) }
\leq C C_\funct r, \qquad r>0.
\end{align}
In addition, letting
\begin{align*}
I_\funct := \int_{\mathbb{C}} \abs{z} \funct(z) \, dA(z),
\end{align*}
the following holds: 
\begin{align}
\label{eq_bv_2}
&\frac{1}{r} \bigg(\abs{B_r(z_0)} - 
\int_{B_r(z_0)} (\funct * 1_{B_r(z_0)})(z) \, dA(z)\bigg) 
\to I_\funct, \qquad 
\mbox{ as } {r \to +\infty}.
\end{align}
\end{lemma}
\begin{proof}
We first note the elementary facts
\begin{align*}
&\abs{B_1(0) \setminus B_1(w)} \leq C \abs{w}, \qquad w \in \mathbb{C},
\\
&\lim_{h \to 0+}
\tfrac{1}{h\abs{w}} \abs{B_1(0) \setminus B_1(h w)} =1,
\end{align*}
for some constant $C>0$. 
For each $w \in \mathbb{C}$,
rescaling and translating yields
\begin{align}
\label{eq_perim_1}
&\abs{B_r(z_0) \setminus B_r(z_0+w)} 
= r^2\bigabs{B_1\big(\tfrac{z_0}{r}\big) \setminus B_1\big(\tfrac{z_0+w}{r}\big)} 
\leq C \abs{w} r, \qquad r > 0,
\\
\label{eq_perim_2}
&\lim_{r \to \infty}
\tfrac{1}{r\abs{w}} \abs{B_r(z_0) \setminus B_r(z_0+w)} =1.
\end{align}
We calculate
\begin{align*}
&\abs{B_r(z_0)} - 
\int_{B_r(z_0)} (\funct * 1_{B_r(z_0)})(z) dA(z)
\\
&\quad
=\int_{B_r(z_0)} 1_{B_r(z_0)}(z) \int_{\mathbb{C}}
\funct(w) \, dA(w) \, dA(z)
- \int_{B_r(z_0)} \int_{\mathbb{C}}
1_{B_r(z_0)}(z-w)
 \funct(w) \, dA(w) \, dA(z)
\\
&\quad=\int_{B_r(z_0)} \int_{\mathbb{C}}
\big(
1_{B_r(z_0)}(z)-1_{B_r(z_0+w)}(z)
\big) \funct(w)  \, dA(w) \, dA(z)
\\
&\quad=
\int_{\mathbb{C}}
\funct(w)
\abs{B_r(z_0) \setminus B_r(z_0+w)} \,dA(w)
\end{align*}
where we used Fubini's theorem.
For \eqref{eq_bv_1}, we use \eqref{eq_perim_1} and estimate
\begin{align*}
\biggabs{
\int_{\mathbb{C}}
\funct(w)
\bigabs{B_r(z_0) \setminus B_r(z_0+w)} \,dA(w) }
\leq C
\int_{\mathbb{C}}
\abs{\funct(w)} \abs{w} r \,dA(w) = C C_\varphi r.
\end{align*}
For \eqref{eq_bv_2}, we use \eqref{eq_perim_2} to obtain
\begin{align*}
\frac{1}{r} \bigg(\abs{B_r(z_0)} - 
\int_{B_r(z_0)} (\funct * 1_{B_r(z_0)})(z) \, dA(z)\bigg) 
=
\int_{\mathbb{C}}
\funct(w) \abs{w} \frac{1}{r\abs{w}}
\abs{B_r(z_0) \setminus B_r(z_0+w)} \,dA(w)
\to I_\funct,
\end{align*}
as $r \to +\infty$, where we used the dominated convergence
theorem, as allowed by \eqref{eq_perim_1}.
\end{proof}

\subsection{Calculation 1}\label{sec_calc_1}
The following calculations can be followed in the symbolic worksheet available at \url{https://github.com/gkoliander/gwhf}.
We wish to calculate
\begin{align*}
E
& = -\frac{1}{2} \Re \Big[ 
\Omega_{1,2}\Omega_{3,4}
+
\Omega_{1,4}\Omega_{3,2}
- 
\Omega_{2,1}\Omega_{3,4}
- 
\Omega_{2,4}\Omega_{3,1}
\Big]
\\
& = 
-\frac{1}{2} \Re \Big[ 
(\Omega_{1,2} - \Omega_{2,1}) \Omega_{3,4}
+
\Omega_{1,4}\Omega_{3,2}
- 
\Omega_{2,4}\Omega_{3,1}
\Big]
\\
& = 
 \Im \big[ \Omega_{1,2} \big] \Im \big[ \Omega_{3,4} \big]
-\frac{1}{2} \Re 
\Big[\Omega_{1,4}\Omega_{3,2}
- 
\Omega_{2,4}\Omega_{3,1}
\Big],
\end{align*}
With the notation of the proof of Proposition \ref{prop_variance_radial}, note that
\begin{align*}
  \Omega_{k,l}  
  = 
  A_{k,l}
  - 
  \frac{(B_{k,1}- e^{- i (yu-xv)}  P B_{k,2}) \overline{B_{l,1}}
  +
  (-e^{i (yu-xv)} P B_{k,1} + B_{k,2}) \overline{B_{l,2}}}{1- P ^2}
\end{align*}
Inserting the various specific values, we obtain
\begin{align*}
  \Omega_{1,2} 
  & = 
  \Gamma(z)_{2,3}
  -
  \frac{(B_{1,1}- e^{- i (yu-xv)}  P B_{1,2}) \overline{B_{2,1}}
  +
  (-e^{i (yu-xv)} P B_{1,1} + B_{1,2}) \overline{B_{2,2}}}{1- P ^2}
  \\
  & = 
  -i-xy
  - 
  \frac{- xy+  xvP^2 + 2ix (x-u)PP'
  +
  (iy P - iv P + 2(x-u)P') ( -iu P + 2(y-v)P')}{1- P ^2}
  \\
  & = 
  \Re \big[ \Omega_{1,2} \big]
  + i 
  \bigg(
  \frac{
  - 2 r^2 PP'
  }{1- P ^2}
  -1
  \bigg)
\end{align*}
\begin{align*}
  \Omega_{3,4} 
  & = 
  \Gamma(w)_{2,3}
  -
  \frac{(B_{3,1}- e^{- i (yu-xv)}  P B_{3,2}) \overline{B_{4,1}}
  +
  (-e^{i (yu-xv)} P B_{3,1} + B_{3,2}) \overline{B_{4,2}}}{1- P ^2}
  \\
  & = 
  -i-uv
  - 
  \frac{(  - iy P - 2(x-u)P' +  iv P )(-ix P - 2(y-v)P')
  +
  uy P^2 - iu 2(x-u)PP'-uv }{1- P ^2}
  \\
  & = 
  \Re \big[ \Omega_{3,4} \big]
  + i 
  \bigg(
  \frac{
  - 2 r^2 PP'
  }{1- P ^2}
  -1
  \bigg)
\end{align*}
Thus,
\begin{align*}
  \Omega_{1,4} 
  & = 
  \Gamma(z,w)_{2,3}
  -
  \frac{(B_{1,1}- e^{- i (yu-xv)}  P B_{1,2}) \overline{B_{4,1}}
  +
  (-e^{i (yu-xv)} P B_{1,1} + B_{1,2}) \overline{B_{4,2}}}{1- P ^2}
  \\
  & = 
  e^{i (yu-xv)}
  \bigg((-i-xv)P +2 i (v(y-v)-x(x-u))P' -4 (x-u)(y-v)P''
  \\
  & \quad 
  - 
  \frac{(-iy +  iv P^2 - 2(x-u)PP' )   \big(-ix P - 2(y-v)P'\big)
  +
   u(y-v) P  - 2iu(x-u)P'}{1- P^2}
  \bigg)
  \\
  & = 
  \frac{e^{i (yu-xv)}}{1- P^2}
  \Big((x-u) (y-v) (4 P^2P''-4 PP'^2+P-4P'')
  + i \big(
  -2r^2 P'  + P^3-P
  \big)
  \Big)
\end{align*}
\begin{align*}
  \Omega_{3,2} 
  & = 
  \overline{\Gamma(z,w)_{3,2}}
    - 
  \frac{(B_{3,1}- e^{- i (yu-xv)}  P B_{3,2}) \overline{B_{2,1}}
  +
  (-e^{i (yu-xv)} P B_{3,1} + B_{3,2}) \overline{B_{2,2}}}{1- P^2}
  \\
  & = 
  e^{-i (yu-xv)}
  \bigg((-i-uy)P -2 i (y(y-v)-u(x-u))P' -4 (x-u)(y-v)P''
  \\
  & \quad 
  - 
  \frac{ - x(y-v) P  + 2 ix(x-u)P'
  +
  (  iy P^2 + 2(x-u)PP' -iv)  \big( -iu P + 2(y-v)P'\big)}{1- P^2}
  \bigg)
  \\
  & = 
  \frac{e^{-i (yu-xv)}}{1- P^2}
  \Big((x-u) (y-v) (4 P^2P''-4 PP'^2+P-4P'')
  + i \big(
  -2r^2 P'  + P^3-P
  \big)
  \Big)
\end{align*}
and the real part of the product is given as
\begin{equation}
  \Re \big[
\Omega_{1,4}\Omega_{3,2}\big]
= \frac{(x-u)^2 (y-v)^2 (Q+P)^2 - 
  (2r^2 P'  - P^3+P)^2}{(1- P^2)^2}
\end{equation}
where we substituted $Q = 4P^2P''-4PP'^2-4P''$.
Similarly, 
\begin{align*}
  \Omega_{2,4} 
  & = 
  \Gamma(z,w)_{3,3}
    - 
  \frac{(B_{2,1}- e^{- i (yu-xv)}  P B_{2,2}) \overline{B_{4,1}}
  +
  (-e^{i (yu-xv)} P B_{2,1} + B_{2,2}) \overline{B_{4,2}}}{1- P^2}
  \\
  & = 
  e^{i (yu-xv)}
  \bigg(xu P - (2+ 2 i (y-v)(x+u))P' +4 (y-v)^2 P''
  \\*
  & \quad 
  - 
  \frac{(ix -     iu P^2 - 2(y-v)P P' )  \big(-ix P - 2(y-v)P'\big)
  - u(x-u) P - 2iu(y-v)P' }{1- P^2}
  \bigg)
  \\
  & = 
  -\frac{e^{i (yu-xv)}}{1- P^2}
  \Big(
    (x-u)^2 P - 4(y-v)^2 (P^2 P''-P P'^2- P'') + 2 P' (1-P^2)
  \Big)
\end{align*}
\begin{align*}
  \Omega_{3,1} 
  & = 
  \overline{\Gamma(z,w)_{2,2}}
    - 
  \frac{(B_{3,1}- e^{- i (yu-xv)}  P B_{3,2}) \overline{B_{1,1}}
  +
  (-e^{i (yu-xv)} P B_{3,1} + B_{3,2}) \overline{B_{1,2}}}{1- P^2}
  \\
  & = 
  e^{-i (yu-xv)}
  \bigg( yv P - (2 -  2 i (x-u)(y+v))P' + 4 (x-u)^2 P''
  \\*
  & \quad 
  - 
  \frac{  y(y-v) P - 2iy(x-u)P'
  +
  ( iy P^2 + 2(x-u)P P' -iv)  ( iv P + 2(x-u)P')}{1- P^2}
  \bigg)
  \\
  & = 
  \frac{e^{-i (yu-xv)}}{1- P ^2}
  \Big(
    -(y-v)^2 P + 4(x-u)^2 (P^2 P''-P P'^2- P'') - 2 P' (1-P^2)
  \Big)
\end{align*}
and the real part of the product is given as
\begin{align*}
  \Re \big[
\Omega_{2,4}\Omega_{3,1}\big]
  & = \frac{
    \big((x-u)^2 P - (y-v)^2 Q + 2 P' (1-P^2)\big)\big((y-v)^2 P - (x-u)^2 Q + 2 P' (1-P^2)\big)
  }{(1- P^2)^2}
  \\
  & = 
  \frac{
    (x-u)^2 (y-v)^2 P^2 - (y-v)^4 PQ + 2 (y-v)^2 PP' (1-P^2)
    }{(1- P^2)^2}
  \\*
  & \quad
  - 
  \frac{
    (x-u)^4 P Q - (y-v)^2(x-u)^2 Q^2 + 2 (x-u)^2 QP' (1-P^2)
    }{(1- P^2)^2}
  \\*
  & \quad
  + 
  \frac{
  2 (x-u)^2 P P' (1-P^2) - 2 (y-v)^2 Q P' (1-P^2)+ 4 P'^2 (1-P^2)^2
  }{(1- P^2)^2}
  \\
  & = 
  \frac{
    (x-u)^2 (y-v)^2 (P^2+Q^2) - ((x-u)^4 + (y-v)^4) PQ + 2 r^2  P'(P-Q) (1-P^2)
    }{(1- P^2)^2}
    + 4 P'^2
  \\
  & = 
  \frac{
    (x-u)^2 (y-v)^2 (P+Q)^2 - r^4 PQ + 2 r^2  P'(P-Q) (1-P^2)
    }{(1- P^2)^2}
    + 4 P'^2
\end{align*}
Combining everything, we obtain 
\begin{align*}
E
& = 
 \Im \big[ \Omega_{1,2} \big] \Im \big[ \Omega_{3,4} \big]
-\frac{1}{2} \Re 
\Big[\Omega_{1,4}\Omega_{3,2}
- 
\Omega_{2,4}\Omega_{3,1}
\Big]
\\
& = 
\bigg(
  \frac{
   2 r^2 PP'
  }{1- P^2}
  +1
  \bigg)^2
  -
  \frac{(x-u)^2 (y-v)^2 (Q+P)^2 - 
  (2r^2 P'  - P^3+P)^2}{2(1- P^2)^2}
  \\*
  & \quad 
  +
  \frac{
    (x-u)^2 (y-v)^2 (P+Q)^2 - r^4 PQ + 2 r^2  P'(P-Q) (1-P^2)
    }{2(1- P ^2)^2}
    + 2 P'^2
\\
& = 
\bigg(
  \frac{
   2 r^2 PP'
  }{1- P^2}
  +1
  \bigg)^2
  +
  \frac{
  (2r^2 P'  + P (1-P^2))^2
  - r^4 PQ + 2 r^2  P'(P-Q) (1-P^2)
  }{2(1- P ^2)^2}
  + 2 P'^2
\\
& = 
  \frac{
  8 r^4 P^2 P'^2 
  +  4r^4 P'^2    
  - r^4 PQ 
  }{2(1- P ^2)^2}
  + 
  \frac{
    r^2  P'(7P-Q)
  }{1- P^2}
  +
  1 + 2 P'^2
  +\frac{P^2}{2}
\\
& = 
  \frac{2r^4
  (3 P^2 P'^2 
  +  P'^2    
  + PP''(1-P^2))
  }{(1- P^2)^2}
  + 
  \frac{
    r^2  P'(7P+4PP'^2+4P''(1- P^2))
  }{1- P ^2}
  +
  1 + 2 P'^2
  +\frac{P^2}{2}
\end{align*}
On the other hand, we have for $I$ defined in \eqref{eq_I} that
\begin{align*}
I' 
&= 
\frac{(2P'^2+\frac{3}{2} P^2)(1-P^2) + r^2 (4P'P''+3 PP')(1-P^2) + 2 r^2 (2P'^2+\frac{3}{2} P^2)PP'}{(1-P^2)^2}
  \\*
  & \quad 
+ \frac{4 r^2 PP'(1-P^2)^2 + 2 r^4 (P'^2+PP'')(1-P^2)^2
+ 8 r^4 PP' (1-P^2)PP'}{(1-P^2)^4}
\\
& = 
\frac{2 r^4 ( 3P^2P'^2+ P'^2+PP''(1-P^2))}{(1-P^2)^3}
+
\frac{r^2P' ( 7 P + 4PP'^2+4P''(1-P^2))}{(1-P^2)^2}
+ 
\frac{2P'^2+\frac{3}{2} P^2 }{1-P^2}
\end{align*}
and see that $E/(1-P^2) -1= I'$.

Finally, we verify \eqref{eq_bound_I}.
By \eqref{eq_vvvv},
\begin{align*}
\lim_{r \rightarrow \infty} 
\frac{1}{1- P^2(r^2)} = 1.
\end{align*}
Inspection of each term in the other factor in $I'$ combined with \eqref{eq_vvvv}
shows that
\[\limsup_{r \rightarrow \infty} r^4 \abs{I'(r^2)} < \infty.\]
Since $I'$ is continuous, it follows that
\begin{align*}
\sup_{r \geq 0} (1+r^4)\abs{I'(r^2)} < \infty,
\end{align*}
as claimed.
\qed


\begin{thebibliography}{10}
	
	\bibitem{adler}
	R.~J. Adler and J.~E. Taylor.
	\newblock {\em Random fields and geometry}.
	\newblock Springer Monographs in Mathematics. Springer, New York, 2007.
	
	\bibitem{MR2593994}
	S.~T. Ali, F.~Bagarello, and G.~Honnouvo.
	\newblock Modular structures on trace class operators and applications to
	{L}andau levels.
	\newblock {\em J. Phys. A}, 43(10):105202, 17, 2010.
	
	\bibitem{level}
	J.-M. Aza\"{\i}s and M.~Wschebor.
	\newblock {\em Level sets and extrema of random processes and fields}.
	\newblock John Wiley \& Sons, Inc., Hoboken, NJ, 2009.
	
	\bibitem{balk}
	M.~B. Balk.
	\newblock Polyanalytic functions.
	\newblock In {\em Complex analysis}, volume~61 of {\em Math. Lehrb\"{u}cher
		Monogr. II. Abt. Math. Monogr.}, pages 68--84. Akademie-Verlag, Berlin, 1983.
	
	\bibitem{MR4047541}
	R.~Bardenet, J.~Flamant, and P.~Chainais.
	\newblock On the zeros of the spectrogram of white noise.
	\newblock {\em Appl. Comput. Harmon. Anal.}, 48(2):682--705, 2020.
	
	\bibitem{bh}
	R.~Bardenet and A.~Hardy.
	\newblock Time-frequency transforms of white noises and {G}aussian analytic
	functions.
	\newblock {\em Appl. Comput. Harmon. Anal.}, 50:73--104, 2021.
	
	\bibitem{MR489542}
	M.~V. Berry.
	\newblock Regular and irregular semiclassical wavefunctions.
	\newblock {\em J. Phys. A}, 10(12):2083--2091, 1977.
	
	\bibitem{MR1913853}
	M.~V. Berry.
	\newblock Statistics of nodal lines and points in chaotic quantum billiards:
	perimeter corrections, fluctuations, curvature.
	\newblock {\em J. Phys. A}, 35(13):3025--3038, 2002.
	
	\bibitem{berry2000phase}
	M.~V. Berry and M.~R. Dennis.
	\newblock Phase singularities in isotropic random waves.
	\newblock {\em R. Soc. Lond. Proc. Ser. A Math. Phys. Eng. Sci.},
	456(2001):2059--2079, 2000.
	
	\bibitem{MR662182}
	L.~Blum, C.~Gruber, J.~L. Lebowitz, and P.~Martin.
	\newblock Perfect screening for charged systems.
	\newblock {\em Phys. Rev. Lett.}, 48(26):1769--1772, 1982.
	
	\bibitem{bs93}
	S.~Brekke and K.~Seip.
	\newblock Density theorems for sampling and interpolation in the
	{B}argmann-{F}ock space. {III}.
	\newblock {\em Math. Scand.}, 73(1):112--126, 1993.
	
	\bibitem{MR2104882}
	M.~R. Douglas, B.~Shiffman, and S.~Zelditch.
	\newblock Critical points and supersymmetric vacua. {I}.
	\newblock {\em Comm. Math. Phys.}, 252(1-3):325--358, 2004.
	
	\bibitem{ehr20}
	L.~A. Escudero, A.~Haimi, and J.~L. Romero.
	\newblock Multiple sampling and interpolation in weighted {F}ock spaces of
	entire functions.
	\newblock {\em Complex Anal. Oper. Theory}, 15(2):Paper No. 35, 32, 2021.
	
	\bibitem{MR3937291}
	R.~Feng.
	\newblock Correlations between zeros and critical points of random analytic
	functions.
	\newblock {\em Trans. Amer. Math. Soc.}, 371(8):5247--5265, 2019.
	
	\bibitem{flandrin2015time}
	P.~Flandrin.
	\newblock Time--frequency filtering based on spectrogram zeros.
	\newblock {\em IEEE Signal Process. Lett.}, 22(11):2137--2141, 2015.
	
	\bibitem{flandrin2018explorations}
	P.~Flandrin.
	\newblock {\em Explorations in time-frequency analysis}.
	\newblock Cambridge University Press, 2018.
	
	\bibitem{folland89}
	G.~B. Folland.
	\newblock {\em Harmonic analysis in phase space}, volume 122 of {\em Annals of
		Mathematics Studies}.
	\newblock Princeton University Press, Princeton, NJ, 1989.
	
	\bibitem{gh13}
	A.~Ghanmi.
	\newblock Operational formulae for the complex {H}ermite polynomials
	{$H_{p,q}(z,\overline{z})$}.
	\newblock {\em Integral Transforms Spec. Funct.}, 24(11):884--895, 2013.
	\newblock Typos corrected in: arXiv:1211.5746v3.
	
	\bibitem{ghosal2006posterior}
	S.~Ghosal and A.~Roy.
	\newblock Posterior consistency of {G}aussian process prior for nonparametric
	binary regression.
	\newblock {\em Ann. Statist.}, 34(5):2413--2429, 2006.
	
	\bibitem{ghosh2017fluctuations}
	S.~Ghosh and J.~L. Lebowitz.
	\newblock Fluctuations, large deviations and rigidity in hyperuniform systems:
	a brief survey.
	\newblock {\em Indian J. Pure Appl. Math.}, 48(4):609--631, 2017.
	
	\bibitem{charlybook}
	K.~Gr\"{o}chenig.
	\newblock {\em Foundations of time-frequency analysis}.
	\newblock Applied and Numerical Harmonic Analysis. Birkh\"{a}user Boston, Inc.,
	Boston, MA, 2001.
	
	\bibitem{gjm20}
	K.~Gr\"{o}chenig, P.~Jaming, and E.~Malinnikova.
	\newblock Zeros of the {W}igner distribution and the short-time {F}ourier
	transform.
	\newblock {\em Rev. Mat. Complut.}, 33(3):723--744, 2020.
	
	\bibitem{MR3336091}
	K.~Gr\"{o}chenig, J.~Ortega-Cerd\`a, and J.~L. Romero.
	\newblock Deformation of {G}abor systems.
	\newblock {\em Adv. Math.}, 277:388--425, 2015.
	
	\bibitem{gafbook}
	J.~B. Hough, M.~Krishnapur, Y.~Peres, and B.~Vir\'{a}g.
	\newblock {\em Zeros of {G}aussian analytic functions and determinantal point
		processes}, volume~51 of {\em University Lecture Series}.
	\newblock American Mathematical Society, Providence, RI, 2009.
	
	\bibitem{is16}
	M.~E.~H. Ismail.
	\newblock Analytic properties of complex {H}ermite polynomials.
	\newblock {\em Trans. Amer. Math. Soc.}, 368(2):1189--1210, 2016.
	
	\bibitem{jan03}
	A.~J. E.~M. Janssen.
	\newblock Zak transforms with few zeros and the tie.
	\newblock In {\em Advances in {G}abor analysis}, Appl. Numer. Harmon. Anal.,
	pages 31--70. Birkh\"{a}user Boston, Boston, MA, 2003.
	
	\bibitem{le83}
	J.~L. Lebowitz.
	\newblock Charge fluctuations in coulomb systems.
	\newblock {\em Phys. Rev. A}, 27:1491--1494, Mar 1983.
	
	\bibitem{martin1980charge}
	P.~A. Martin and T.~Yalcin.
	\newblock The charge fluctuations in classical {C}oulomb systems.
	\newblock {\em J. Statist. Phys.}, 22(4):435--463, 1980.
	
	\bibitem{NSwhat}
	F.~Nazarov and M.~Sodin.
	\newblock What is{$\ldots$}a {G}aussian entire function?
	\newblock {\em Notices Amer. Math. Soc.}, 57(3):375--377, 2010.
	
	\bibitem{torquato2016hyperuniformity}
	S.~Torquato.
	\newblock Hyperuniformity and its generalizations.
	\newblock {\em Phys. Rev. E}, 94(2):022122, 2016.
	
	\bibitem{MR3815253}
	S.~Torquato.
	\newblock Hyperuniform states of matter.
	\newblock {\em Phys. Rep.}, 745:1--95, 2018.
	
	\bibitem{trifonov2001schrodinger}
	D.~Trifonov.
	\newblock Schr{\"o}dinger uncertainty relation and its minimization states.
	\newblock {\em Phys.World}, 24:107--116, 2001.
	
	\bibitem{vas00}
	N.~L. Vasilevski.
	\newblock Poly-{F}ock spaces.
	\newblock In {\em Differential operators and related topics, {V}ol. {I}
		({O}dessa, 1997)}, volume 117 of {\em Oper. Theory Adv. Appl.}, pages
	371--386. Birkh\"{a}user, Basel, 2000.
	
	\bibitem{wilkinson2004screening}
	M.~Wilkinson.
	\newblock Screening of charged singularities of random fields.
	\newblock {\em J. Phys. A}, 37(26):6763--6771, 2004.
	
	\bibitem{zhu}
	K.~Zhu.
	\newblock {\em Analysis on {F}ock spaces}, volume 263 of {\em Graduate Texts in
		Mathematics}.
	\newblock Springer, New York, 2012.
	
\end{thebibliography}
\end{document}